\documentclass[10pt,leqno,letterpaper]{article}
\usepackage[utf8]{inputenc}
\usepackage{amsmath}
\usepackage{amsfonts}
\usepackage{amssymb}
\usepackage{graphicx}
\usepackage{verbatim}
\usepackage{mathrsfs}
\usepackage{upref,amsthm,amsxtra,exscale}
\usepackage{cite}
\usepackage{fullpage}
\usepackage[colorlinks=true,urlcolor=blue,
citecolor=red,linkcolor=blue,linktocpage,pdfpagelabels,
bookmarksnumbered,bookmarksopen]{hyperref}
\usepackage{upgreek}

\newtheorem{theorem}{Theorem}[section]
\newtheorem{corollary}[theorem]{Corollary}
\newtheorem{remark}[theorem]{Remark}

\newtheorem{lemma}[theorem]{Lemma}
\newtheorem{proposition}[theorem]{Proposition}

\newtheorem{example}[theorem]{Example}
\newtheorem{examples}[theorem]{Examples}

\newtheorem{notation}[theorem]{Notation}

\numberwithin{equation}{section}

\def\r{\mathbb{R}}
\def\rn{\mathbb{R}^N}
\def\z{\mathbb{Z}}

\def\n{\mathbb{N}}
\def\cc{\mathbb{C}}
\def\eps{\varepsilon}
\def\rh{\rightharpoonup}
\def\io{\int_{\Omega}}
\def\irn{\displaystyle\int_{\r^N}}
\def\vp{\varphi}
\def\o{\Omega}
\def\t{\Theta}
\def\bf{\boldsymbol}
\def\wt{\widetilde}
\def\cC{\mathcal{C}}

\def\cH{\mathcal{H}}
\def\cI{\mathcal{I}}
\def\cJ{\mathcal{J}}
\def\cK{\mathcal{K}}
\def\cM{\mathcal{M}}
\def\cN{\mathcal{N}}

\def\cU{\mathcal{U}}

\def\F{\mathrm{Fix}}
\def\supp{\text{supp}}
\def\bar{\overline}

\def\d{\,\mathrm{d}}

\newcommand{\dlim}{\displaystyle\lim}
\newcommand{\dmin}{\displaystyle\min}
\newcommand{\dmax}{\displaystyle\max}

\newcommand{\dsum}{\displaystyle\sum}
\newcommand{\what}{\widehat}

\author{Mónica Clapp\footnote{M. Clapp was supported by CONACYT (Mexico) through the grant A1-S-10457.}\qquad and \qquad Mayra Soares\footnote{M. Soares was supported by UNAM-DGAPA (Mexico) through a postdoctoral fellowship.}}
\title{Energy estimates for seminodal solutions to an elliptic system with mixed couplings}
\date{\today}

\begin{document}
	\maketitle
	
\begin{abstract}
We study the system of semilinear elliptic equations		
$$-\Delta u_i+ u_i = \dsum_{j=1}^\ell \beta_{ij}|u_j|^p|u_i|^{p-2}u_i, \qquad u_i\in H^1(\rn),\qquad i=1,\ldots,\ell,$$
where $N\geq 4$, $1<p<\frac{N}{N-2}$, and the matrix $(\beta_{ij})$ is symmetric and admits a block decomposition such that the entries within each block are positive or zero and all other entries are negative.

We provide simple conditions on $(\beta_{ij})$, which guarantee the existence of fully nontrivial solutions, i.e., solutions all of whose components are nontrivial. 

We establish existence of fully nontrivial solutions to the system having a prescribed combination of positive and nonradial sign-changing components, and we give an upper bound for their energy when the system has at most two blocks. 

We derive the existence of solutions with positive and nonradial sign-changing components to the system of singularly perturbed elliptic equations
$$-\eps^2\Delta u_i+ u_i = \dsum_{j=1}^\ell \beta_{ij}|u_j|^p|u_i|^{p-2}u_i, \qquad u_i\in H^1_0(B_1(0)),\qquad i=1,\ldots,\ell,$$
in the unit ball, exhibiting two different kinds of asymptotic behavior: solutions whose components decouple as $\eps\to 0$, and solutions whose components remain coupled all the way up to their limit.
\medskip

\medskip
		
	
\end{abstract}
	
\section{Introduction}
	
Consider the system of nonlinear elliptic equations
\begin{equation} \label{eq:system}
\begin{cases}
-\Delta u_i+ u_i = \dsum_{j=1}^\ell \beta_{ij}|u_j|^p|u_i|^{p-2}u_i, \\
u_i\in H^1(\rn),\qquad i=1,\ldots,\ell,
\end{cases}
\end{equation}
where $N\geq 4$, $\beta_{ij}\in\r$, and $1<p<\frac{N}{N-2}$.

This system arises as a model for various physical phenomena, for instance, in the study of standing	waves for a mixture of Bose-Einstein condensates of $\ell$ different hyperfine states which overlap in space; see  \cite{EGBJB}. The coefficient $\beta_{ij}$ represents the interaction force between the components $u_i$ and $u_j$. It can be attractive ($\beta_{ij} > 0$), repulsive ($\beta_{ij} < 0$) or null ($\beta_{ij}=0$). We make the following assumptions:
\begin{itemize}
\item[$(B_1)$] The matrix $(\beta_{ij})$ is symmetric and admits a block decomposition as follows: For some $1 \leq q \leq\ell$ there exist \ $0=\ell_0<\ell_1<\dots<\ell_{q-1}<\ell_q=\ell$ \ such that, if we set
\begin{align*}
	&I_h:= \{i \in  \{1,\dots,\ell\}:  \ell_{h-1} < i \le \ell_h \},\\
	&\cI_h:=I_h\times I_h,\quad \text{and} \quad \cK_h:=\big\{(i,j)\in I_h\times I_k: k\in\{1,\ldots,q\}\smallsetminus\{h\}\big\},
\end{align*}
then $\beta_{ii}>0, \ \beta_{ij}\geq 0 \text{ for all } (i,j)\in \cI_h,$ \ and \ $\beta_{ij}<0 \text{ for all } (i,j)\in \cK_h,\quad h=1,\ldots,q$.
\item[$(B_2)$] For each $h=1,\ldots,q$, the graph whose set of vertices is $I_h$ and whose set of edges is $E_h:=\{\{i,j\}:i,j\in I_h, \ i\neq j, \ \beta_{ij}>0\}$ is connected. 
\end{itemize}

According to the decomposition given by $(B_1)$, a solution $\bf u=(u_1,\ldots,u_\ell)$ to \eqref{eq:system} may be written in block-form as
\[\bf u=(\bar u_1,\ldots,\bar u_q)\qquad\text{with \ }\bar u_h=(u_{\ell_{h-1}+1},\ldots,u_{\ell_h}).\]
$\bf u$ is said to be \emph{nontrivial} if at least one of its components $u_i$ is different from zero and it is said to be \emph{fully nontrivial} if every component $u_i$ is different from zero. We shall call it \emph{block-wise nontrivial} if at least one component in each block $\bar u_h$ is nontrivial.

If $q=1$ the system is called \emph{cooperative} and it is called \emph{competitive} if $q=\ell$. These kinds of systems have been studied extensively. Systems with mixed couplings were already investigated in the seminal paper by Lin and Wei \cite{LW2}, and more recently, e.g.,  in \cite{bsw, cp, dp, sw1, sw2, so, st, ty, tyz, ww}.

Our main interest is to study the existence of fully nontrivial solutions to the system \eqref{eq:system} with a prescribed set of positive and sign-changing components - referred to in the literature as \emph{seminodal} or \emph{semipositive} solutions -, and to give an upper bound for their energy.

We stress that obtaining an upper bound for the least energy of a solution to the system to \eqref{eq:system} is interesting even in the case when all components are positive. Indeed, as shown by Lin and Wei in \cite{LW2} a positive least energy solution does not exist when $q\geq 2$; see also Proposition \ref{prop:ground_states} below. On the other hand, in \cite{s} Sirakov established the existence of solutions with positive radial components, while sign-changing and seminodal solutions have been exhibited, e.g., in \cite{clz, CS, llw, sw3, sw4, tt}. But nothing is said about their energy. For a single equation an upper bound for the least energy of a sign-changing solution is given in \cite{CSr}. Here we improve that upper bound; see Corollary \ref{cor:cs} below.

We write $\|\,\cdot\,\|$ for the usual norm in the Sobolev space $H^1(\rn)$ and, abusing notation, we set
\[\|\bf u\|^2:=\sum_{i=1}^\ell\|u_i\|^2\qquad\text{for \ }\bf u=(u_1,\ldots,u_\ell)\in (H^1(\rn))^\ell.\]
For each $h=1,\ldots,q$ \ set \ $\r^{I_h}:=\{\bar s=(s_{\ell_{h-1}+1},\ldots,s_{\ell_h}):s_i\in\r\text{ for all }i\in I_h\}$ \ and define
\begin{equation} \label{eq:mu_h}
\mu_h:=\inf_{\substack{\bar s\in\r^{I_h} \smallskip \\ \bar s\neq 0}}\left(\frac{\dsum_{i\in I_h}s_i^2}{\Big(\dsum\limits_{{(i,j)\in \cI_h}}\beta_{ij}|s_i|^p|s_j|^p\Big)^\frac{2}{2p}}\right)^\frac{p}{p-1}.
\end{equation}
Our first result reads as follows.

\begin{theorem}\label{thm:Main1}
Assume $(B_1)$, $(B_2)$ and
\begin{itemize}
\item[$(B_3)$] If $q\geq 2$ then, for every $h\in\{1,\ldots,q\}$ such that $\ell_{h}-\ell_{h-1}\geq 2$, the inequality
$$\Big(\min_{\{i,j\}\in E_h}\beta_{ij}\Big)\left[\frac{\dmin_{h=1,\ldots,q}\dmax_{i\in I_h}\beta_{ii}}{\dsum_{(i,j)\in\cI_h}\beta_{ij}}\right]^\frac{p}{p-1}>\,C_q\sum_{(i,j)\in\cK_h}|\beta_{ij}|$$
holds true, where $C_q$ is a positive constant independent of $(\beta_{ij})$. 
\end{itemize}
If $q\geq 2$ assume also that $N\neq 5$. Let $\{1,\ldots,q\}=Q^+\cup Q^-$ with $Q^+\cap Q^-=\emptyset$. Then there exists a fully nontrivial solution $\bf w=(\bar w_1,\ldots,\bar w_q)$ to the system \eqref{eq:system} with the following properties:
\begin{itemize}
\item[$(a)$] Every component of $\bar w_h$ is positive if $h\in Q^+$ and every component of $\bar w_h$ is nonradial and changes sign if $h\in Q^-$.
\item[$(b)$] If $q\geq 2$ then
$$\|\bf w\|^2>\sum_{h\in Q^+}\mu_h\,\|\omega\|^2+2\sum_{k\in Q^-}\mu_k\,\|\omega\|^2.$$
\item[$(c)$] If $q=1$, i.e., if the system is cooperative, then 
\[\|\bf w\|^2=\mu_1\,\|\omega\|^2\text{ \ if \ }Q^+=\{1\}\text{ \ and \ }\|\bf w\|^2<10\,\mu_1\,\|\omega\|^2\text{ \ if \ }Q^-=\{1\}.\]
\item[$(d)$] If $q=2$, then
\begin{equation*}
\|\bf w\|^2< 
\begin{cases}
\min\limits_{\substack{h,k\in Q^+  \\ h\neq k}}(\mu_k+6\mu_h)\,\|\omega\|^2 &\text{if \ }Q^+=\{1,2\}, \\
(\mu_k+12\mu_h)\,\|\omega\|^2 &\text{if \ }k\in Q^+\text{ and }h\in Q^-, \smallskip\\
12(\mu_1+\mu_2)\,\|\omega\|^2 &\text{if \ }Q^-=\{1,2\},
\end{cases}
\end{equation*}
\end{itemize}
where $\omega$ is the unique positive radial solution to the equation
\begin{equation*}
-\Delta w+w=|w|^{2p-2}w,\qquad w\in H^1(\rn).
\end{equation*} 
\end{theorem}

To prove this result we introduce suitable symmetries that produce, by construction, a change of sign on some of the components. Then, using concentration compactness techniques, we carry out a careful analysis of the effect of the symmetries on the behavior of symmetric minimizing sequences for the system \eqref{eq:system}; see Theorem \ref{thm:splitting}. This approach was taken in \cite[Theorem 1.1]{CS} to establish existence of solutions with a prescribed number of positive and sign-changing components for the competitive system \eqref{eq:system} with $q=\ell$. However, the symmetries used there are too restrictive and do not allow energy bounds to be obtained. In this paper we consider general symmetries, admitting finite orbits. This makes it more difficult to show existence, due to the lack of compactness of the variational functional. On the other hand, the compactness condition given in Corollary 2.8 immediately gives energy estimates. But showing that this condition is fulfilled requires, in turn, precise knowledge about the asymptotic decay of the components of the system. This information does not seem to be available so far. This is why the upper energy bounds in Theorem \ref{thm:Main1} are given only for a system with at most two blocks; see Remark \ref{rem:bounds}.

Regarding our assumptions on the matrix $(\beta_{ij})$, condition $(B_1)$ is enough to guarantee the existence of a block-wise nontrivial solution (see Theorem \ref{thm:existence}), while $(B_2)$ and $(B_3)$ ensure that this solution is fully nontrivial. These last two conditions are weaker than the condition stated in \cite[Theorem 1.2]{cp}. Indeed, the latter forces $\beta_{ij}>0$ if $i,j\in I_h$ for every $h=1,\ldots,q$. Note that if this is the case, the graph defined in $(B_2)$ is complete and, therefore, connected. So the condition in \cite[Theorem 1.2]{cp} implies $(B_2)$ and $(B_3)$ but not the other way around. The constant $C_q$ that appears in $(B_3)$ depends on the symmetries we have chosen to obtain Theorem \ref{thm:Main1} and it is explicitly given in its proof.

The following result is a special case, for $\ell=1$, of Theorem \ref{thm:Main1}. It improves \cite[Theorem 1.1]{CSr}.

\begin{corollary} \label{cor:cs}
For $N\geq 4$ the equation 
\begin{equation*}
-\Delta w+w=|w|^{2p-2}w,\qquad w\in H^1(\rn).
\end{equation*} 
has a nonradial sign-changing solution $\widehat{\omega}$ such that $\|\widehat{\omega}\|^2<10\,\|\omega\|^2$, where $\omega$ is the unique positive radial solution to this equation.
\end{corollary}

The system \eqref{eq:system} is the limit, as $\eps\to 0$, of the system of singularly perturbed elliptic equations
\begin{equation}\label{eq:system_eps}
\tag{$\mathscr S_{\eps,\mathbb{B}}$}\qquad
\begin{cases}
-\eps^2\Delta u_i+ u_i = \dsum_{j=1}^\ell \beta_{ij}|u_j|^p|u_i|^{p-2}u_i, \\
u_i\in H^1_0(\mathbb{B}),\qquad i=1,\ldots,\ell,
\end{cases}
\end{equation}
where $\eps>0$, $\mathbb{B}$ is the unit ball in $\rn$, $N\geq 4$, $1<p<\frac{N}{N-2}$, and the matrix $(\beta_{ij})$ satisfies assumptions $(B_1)$ and $(B_2)$. 

In \cite{LW} Lin and Wei described the limit profile of positive least energy solutions to this system. They showed that their components blow up and decouple as $\eps\to 0$. Each component is a rescaling of the positive ground state of the stationary Schrödinger equation in $\rn$ and the concentration points of the components approach a sphere-packing position as $\eps\to 0$. 
The analysis of the effect of symmetries on the behavior of symmetric minimizing sequences for the system \eqref{eq:system} given by Theorem \ref{thm:splitting} allows to exhibit solutions for $(\mathscr S_{\eps,\mathbb{B}})$ that blow up with two different types of limit profiles: solutions whose blocks decouple as $\eps\to 0$ and solutions whose blocks remain coupled all the way up to their limit. The following results extend and complete those obtained in \cite{CS} for competitive systems.

\begin{theorem}[Solutions with decoupling blocks] \label{thm:Main2}
Assume $(B_1)$ and $(B_2)$ and let $N\geq 5$. Given a partition $\{1,\ldots,q\}=Q^+\cup Q^-$ with $Q^+\cap Q^-=\emptyset$ and a sequence of positive numbers $(\eps_n)$ with $\eps_n\to 0$, there exists a fully nontrivial solution ${\bf u}_n=(\bar{u}_{1n},\ldots,\bar{u}_{q n})$ to the system $(\mathscr{S}_{\eps_n, \mathbb{B}})$ such that the components of $\bar{u}_{hn}$ are positive if $h\in Q^+$, are nonradial and change sign if $h\in Q^-$, and have the following limit profile: For each $h=1,\ldots,q$, there exist a sequence $(\zeta_{hn})$ in $\mathbb{B}$ and a least energy fully nontrivial solution $\bar v_h$ to the cooperative system
	\begin{equation}\label{eq:subsystem_h}
		\tag{$\mathscr S_{h}$}\qquad
		\begin{cases}
			-\Delta v_i+ v_i = \dsum\limits_{j\in I_h} \beta_{ij}|v_j|^p|v_i|^{p-2}v_i, \\
			v_i\in H^1(\rn),\qquad i\in I_h,
		\end{cases}
	\end{equation}
such that, after passing to a subsequence,
	\begin{itemize}
		\item[$(a)$] the components of $\bar v_h$ are positive if $h\in Q^+$ and they are nonradial and change sign if $h\in Q^-$,
		\item[$(b)$]  $\dlim_{n\to\infty}\eps_n^{-1}\mathrm{dist}(\zeta_{hn},\partial\mathbb{B})=\infty$ and $\dlim_{n\to\infty}\eps_n^{-1}|\zeta_{hn}-\zeta_{kn}|=\infty$ if $h\neq k$, \ $h,k=1,\ldots,q$,
		\item[$(c)$] $\dlim_{n\to\infty}\left\|\widetilde{u}_{in}-v_i\right\|=0$ \ for every \ $i=1,\ldots,\ell$, \ where $\widetilde{u}_{in}(z):=u_{in}(\eps_nz +\zeta_{hn}),$
		\item[$(d)$] $\|\bar v_h\|^2=\mu_h\,\|\omega\|^2$ \ if $h\in Q^+$, \ and \ $\|\bar v_h\|^2<10\,\mu_h\,\|\omega\|^2$ \ if $h\in Q^-$.
	\end{itemize}
\end{theorem}

\begin{theorem}[Solutions with persisting coupling] \label{thm:Main3}
	Assume $(B_1)-(B_3)$ and let $N=4$ or $N\geq 6$. Given a partition $\{1,\ldots,q\}=Q^+\cup Q^-$ with $Q^+\cap Q^-=\emptyset$ and a sequence of positive numbers $(\eps_n)$ with $\eps_n\to 0$, there exists a fully nontrivial solution ${\bf u}_n=(\bar{u}_{1n},\ldots,\bar{u}_{q n})$ to the system $(\mathscr{S}_{\eps_n, \mathbb{B}})$ such that the components of $\bar{u}_{hn}$ are positive if $h\in Q^+$ and they are nonradial and change sign if $h\in Q^-$, and $\bf u_n$ has the following limit profile: There exists a fully nontrivial solution $\bf w=(\bar w_1,\ldots,\bar w_q)$ to the system \eqref{eq:system} such that, after passing to a subsequence,
	\begin{itemize}
		\item[$(a)$] the components of $\bar w_h$ are positive if $h\in Q^+$ and they are nonradial and change sign if $h\in Q^-$,
		\item[$(b)$] $\dlim_{n\to\infty}\left\|\widetilde{u}_{in}-w_i\right\|=0$ \ for every \ $i=1,\ldots,\ell$, \ where $\widetilde{u}_{in}(z):=u_{in}(\eps_nz),$
		\item[$(c)$] If $q=2$, then $\bf w$ satisfies the energy estimates given in \emph{Theorem} \ref{thm:Main1}$(c)$.
	\end{itemize}
\end{theorem}

This paper is organized as follows: In Section \ref{sec:preliminaries} we introduce the symmetric variational setting, give an accurate description of the effect of symmetries on the behavior of symmetric minimizing sequences for the system \eqref{eq:system}, and derive a condition for the existence of least energy block-wise nontrivial symmetric solutions. Section \ref{sec:fully_nontrivial} is devoted to showing that such solutions are fully nontrivial. In Section \ref{sec:cooperative} we establish existence of positive and nodal solutions for cooperative systems and estimate their energy. Similar results for positive, nodal and seminodal solutions to a system with two blocks are obtained in Section \ref{sec:two_blocks}. In Section \ref{sec:singularly_perturbed} we prove the results concerning the system of singularly perturbed equations.

\section{The profile of minimizing sequences}
\label{sec:preliminaries}

Throughout this section we assume $(B_1)$.
Let $G$ be a closed subgroup of the group $O(N)$ of linear isometries of $\rn$ and denote by $Gx:=\{gx:g\in G\}$ the $G$-orbit of $x\in\rn$. Let $\phi:G\to\z_2:=\{-1,1\}$ be a continuous homomorphism of groups with the following property:
\begin{itemize}
\item[$(A_\phi)$] If $\phi$ is surjective, then there exists $\zeta\in\rn$ such that $(\ker\phi)\zeta\neq G\zeta$,
\end{itemize} 
where $\ker\phi:=\{g\in G:\phi(g)=1\}$. A function $u:\rn\to\r$ is called \emph{$G$-invariant} if it is constant on $Gx$ for every $x\in\rn$ and will be called \emph{$\phi$-equivariant} if
	\[ u(gx)=\phi(g)u(x) \text{ \ for all \ }g\in G, \ x\in\rn.\]
Define 
	\[H^1(\rn)^\phi:=\{u\in H^1(\rn): u\text{ is }\phi\text{-equivariant}\}.\]
Assumption $(A_\phi)$ guarantees that $H^1(\rn)^\phi$ has infinite dimension, see \cite{JMW}. If $\phi\equiv 1$ is the trivial homomorphism, then $H^1(\rn)^\phi$ is the space of $G$-invariant functions in $H^1(\rn)$. On the other hand, if $\phi$ is surjective, then every nontrivial function $u\in H^1(\rn)^\phi$ is nonradial and changes sign.

If $K$ is a closed subgroup of $G$ we write $\phi|K:K\to\z_2$ for the restriction of $\phi$ to $K$. Note that $\phi|K$ satisfies $(A_{\phi|K})$ if $\phi$ satisfies $(A_\phi)$.
	
Recall the notation in assumption $(B_1)$. Fix a closed subgroup $G$ of $O(N)$ and, for each $h=1,\ldots,q$, fix a continuous homomorphism $\phi_h:G\to\z_2$ satisfying $(A_{\phi_h})$. Set $\bf\phi:=(\phi_1,\ldots,\phi_q)$ and define $\phi_i := \phi_h$ for all $i\in I_h$. 

If $Q=\{h_1,\ldots,h_m\}$ is a nonempty ordered subset of $\{1,\ldots,q\}$ and $K$ is a closed subgroup of $G$, we set $I_Q:=\displaystyle\bigcup_{h\in Q}I_h$ and consider the system
\begin{equation}\label{S_Q}
\tag{$\mathscr S^{\bf\phi|K}_Q$} \qquad
\begin{cases}
-\Delta u_i+ u_i = \dsum_{j\in I_Q}\beta_{ij}|u_j|^p|u_i|^{p-2}u_i, \\
u_i\in H^1(\rn)^{\phi_i|K},\qquad i\in I_Q.
\end{cases}
\end{equation}
  Define
\begin{align*}
 \cH^{\bf\phi|K}_Q&:=\prod_{i\in Q}H^1(\rn)^{\phi_i|K},
\end{align*}
where the factors are ordered following the natural order of the indices $i\in Q$. So a point in $\cH^{\bf\phi|K}_Q$ will be written as  
\[\bf u=(\bar u_{h_1},\ldots,\bar u_{h_m})\qquad\text{with \ }\bar u_h=(u_{\ell_{h-1}+1},\ldots,u_{\ell_h}).\] Abusing notation, we denote their norms by
	\[\|\bar u_h\|:=\Big(\dsum_{i\in I_h}\|u_i\|^2\Big)^{1/2}\qquad\text{and}\qquad\|\bf u\|:=\Big(\dsum_{h\in Q}\|\bar u_h\|^2\Big)^{1/2},\]
where $\|u_i\|$ is the standard norm of $u_i$ in $H^1(\rn)$.

Let $\cJ^{\bf\phi|K}_Q:\cH^{\bf\phi|K}_Q\to\r$ be the functional given by 
	\[\cJ^{\bf\phi|K}_Q(\bf u) := \frac{1}{2}\dsum_{i\in I_Q}\|u_i\|^2 - \frac{1}{2p}\dsum_{i,j\in I_Q}\beta_{ij}\irn |u_i|^p|u_j|^p.\]
This functional is of class $\cC^1$ and its critical points are the solutions to the system \eqref{S_Q}. The block-wise nontrivial ones belong to the set
	\[\cN^{\bf\phi|K}_Q:= \{\bf u\in\cH^{\bf\phi|K}_Q:\|\bar u_h\|\neq 0\text{ \ and \ }\partial_{\bar u_h}\cJ^{\bf\phi|K}_Q(\bf u)\bar u_h=0 \text{ \ for all \ } h\in Q\}.\]
Note that, for each $h\in Q$
\begin{align*}
&\partial_{\bar u_h}\cJ^{\bf\phi|K}_Q(\bf u)\bar u_h=\|\bar u_h\|^2 - \dsum_{k\in Q}\,\dsum_{(i,j)\in I_h\times I_k}\irn\beta_{ij}|u_i|^p|u_j|^p,
\end{align*}
and that
	\[\cJ^{\bf\phi|K}_Q(\bf u)= \dfrac{p-1}{2p}\|\bf u\|^2\qquad\text{if \ }\bf u\in\cN^{\bf\phi|K}_Q.\]
Set
\begin{equation*}
c^{\bf\phi|K}_Q := \inf_{\bf u\in \cN^{\bf\phi|K}_Q}\cJ^{\bf\phi|K}_Q(\bf u).
\end{equation*}

\begin{notation} \label{notation}
If $Q=\{h\}$ we omit the curly brackets and write
\[(\mathscr S^{\bf\phi|K}_h),\qquad \cH^{\bf\phi|K}_h,\qquad \cJ^{\bf\phi|K}_h,\qquad \cN^{\bf\phi|K}_h,\qquad c^{\bf\phi|K}_h.\]
If  $Q=\{1,\ldots,q\}$ we omit the subscript $Q$ and write 
\[(\mathscr S^{\bf\phi|K}),\qquad \cH^{\bf\phi|K},\qquad \cJ^{\bf\phi|K},\qquad \cN^{\bf\phi|K},\qquad c^{\bf\phi|K}.\]
If $K=G$ we write $\bf\phi$ instead of $\bf\phi|G$, if $\phi|K\equiv 1$ we replace the superscript $\phi|K$ by $K$, and if $K$ is the trivial group we omit the superscript.
\end{notation}

For $\bf u=(\bar{u}_{h_1},\ldots,\bar{u}_{h_m})\in\cH^{\bf\phi|K}_Q$ and $\bf s:=(s_{1},\ldots,s_{m})\in(0,\infty)^m$ we set  \[\bf s \bf u:=(s_1\bar{u}_{h_1},\ldots,s_m\bar{u}_{h_m}).\]
Our next result provides helpful information about the Nehari-type set $\cN^{\bf\phi|K}_Q$.
	
\begin{lemma}\label{lem:N}
The following statements hold true:
\begin{itemize}
\item[$(i)$] $\cN^{\bf\phi|K}_Q\neq\emptyset$ and there exists $d_0>0$ such that
\[\min_{1\leq i\leq m}	\|\bar{u}_{h_i}\|^2 > d_0\quad\text{for every \ }\bf u=(\bar{u}_{h_1},\ldots,\bar{u}_{h_m}) \in \cN^{\bf\phi|K}_Q.\]
Therefore $\cN^{\bf\phi|K}_Q$ is a closed subset of $\cH^{\bf\phi|K}_Q$ and $c^{\bf\phi|K}_Q >0$.
\item[$(ii)$] If the coordinates of $\bf u\in \cH^{\bf\phi|K}_Q$ satisfy
\begin{equation} \label{eq:N}
\sum_{k\in Q}\,\sum_{(i,j)\in I_h\times I_k}\irn\beta_{ij}|u_i|^p|u_j|^p>0
\end{equation}
for every $h\in Q$, then there exists a unique $\bf s_{\bf u} \in (0,\infty)^{m},$ such that $\bf s_{\bf u}\bf u \in \cN^{\bf\phi|K}_Q$. Furthermore,  
\[
\cJ^{\bf\phi|K}_Q(\bf s_{\bf u} \bf u) = \max_{{\bf s} \in (0,\infty)^m}\cJ^{\bf\phi|K}_Q(\bf s \bf u).
\]
\item[$(iii)$] There exists $d_{\bf\phi}>0$ independent of $(\beta_{ij})$ such that
$$c^{\bf\phi}\leq d_{\bf\phi}\Big(\min_{h=1,\ldots,q}\max_{i\in I_h}\beta_{ii}\Big)^{-\frac{1}{p-1}}.$$
\end{itemize}
\end{lemma}	
	
\begin{proof}
These results are obtained by adapting the proofs of \cite[Lemma 2.2]{cp} and \cite[Lemma 2.3]{CS} in the obvious way. The constant in statement $(iii)$ is defined as
$$d_{\bf\phi}:=\dfrac{p-1}{2p}\inf_{(v_1,\ldots,v_q)\in\cU^{\bf\phi}}\sum_{h=1}^q\|v_h\|^2,$$
where \ $\cU^{\bf\phi}:=\{(v_1,\ldots,v_q):v_h\in H^1(\rn)^{\bf{\phi}_h}, \ v_h\neq 0, \ \|v_h\|^2=\irn|v_h|^{2p}, \ v_hv_k=0\text{ if }h\neq k\}$.
\end{proof} 

It is shown in \cite[Lemma 2.4]{cp} that any minimizer of $\cJ^{\bf\phi|K}_Q$ on $\cN^{\bf\phi|K}_Q$ is a critical point of $\cJ^{\bf\phi|K}_Q$, i.e., a block-wise nontrivial solution to \eqref{S_Q}. We call it a \textbf{least energy block-wise nontrivial solution of \eqref{S_Q}}.

\begin{remark} \label{rem:positive}
\emph{If $\phi_h\equiv 1$ for every $h$ in some subset $\overline Q$ of $Q$ and $\bf u$ is a least energy block-wise nontrivial solution to \eqref{S_Q}, then $|u_i|\in H^1(\rn)^{K}$ for every $i\in I_{\overline Q}$ and replacing $u_i$ by $|u_i|$ we obtain a least energy block-wise nontrivial solution to \eqref{S_Q} whose $i$-th component is positive for every $i\in I_{\overline Q}$.}
\end{remark}
	
Let $G_\xi:=\{g \in G:g\xi = \xi\}$ be the $G$-isotropy subgroup of the point $\xi\in\rn$. Recall that the $G$-orbit $G\xi$ of $\xi$ is $G$-homeomorphic to the homogeneous space $G/G_\xi$. So they have the same cardinality, i.e., $|G/G_\xi| = |G\xi|$.
	
\begin{lemma} \label{lem:energy_estimates2}
If $Q=Q_1\cup Q_2$ with $Q_1\cap Q_2=\emptyset$ and $\xi_1,\xi_2\in\rn$ are such that $G\xi_1\neq G\xi_2$, then
		\[
		c^{\bf\phi}_Q\leq |G\xi_1|\, c^{\bf\phi|G_{\xi_1}}_{Q_1}+|G\xi_2|\, c^{\bf\phi|G_{\xi_2}}_{Q_2}.
		\]
\end{lemma}	
	
\begin{proof}
To simplify notation, assume $Q_1=\{1,\ldots,m\}$ and $Q_2=\{m+1,\ldots,q\}$ with $1\leq m<q$. The argument is easily adapted to the general case.

If either $|G\xi_1|=\infty$ or $|G\xi_2|=\infty$, the statement is obvious. So let us assume that both $G$-orbits are finite. Let $\bf u_1=(\bar u_1,\ldots,\bar u_m)\in\cN^{\bf\phi|G_{\xi_1}}_{Q_1}$ and $\bf u_2=(\bar u_{m+1},\ldots,\bar u_q)\in\cN^{\bf\phi|G_{\xi_2}}_{Q_2}$. Fix $r>0$ such that
		\[
		|z_1-z_2|>2r\qquad\text{if \ }z_1,z_2\in G\xi_1\cup G\xi_2\text{ \ and \ } z_1\neq z_2.
		\]
		Let $\chi\in\cC_c^\infty(\rn)$ be a radial cut-off function such that $\chi(x)=1$ if $|x|\leq\frac{r}{2}$ and $\chi(x)=0$ if $|x|\geq r$, and for each $n\in\n$ define $\bf u_{1n}:=(\bar u_{1n},\ldots,\bar u_{mn})$ and $\bf u_{2n}:=(\bar u_{m+1\,n},\ldots,\bar u_{q n})$ by setting
		\begin{equation*}
			\bar u_{hn}(x):=\chi\left(\frac{x-n\xi_\nu}{n}\right)\bar u_h(x-n\xi_\nu)\qquad\text{for \ }h\in Q_\nu,\quad \nu=1,2.
		\end{equation*}
Then, $\bar u_{hn}\in \cH^{\bf\phi|G_{\xi_\nu}}_{h}$ and, as
\[\sum_{k\in Q_\nu}\sum_{(i,j)\in I_h\times I_k}\irn\beta_{ij}|u_i|^p|u_j|^p=\|\bar u_h\|^2\geq d_0>0,\]
Lemma \ref{lem:N} \ yields $\bf t^1_{n}=(t_{1n},\ldots,t_{mn})\in(0,\infty)^m$ and $\bf t^2_{n}=(t_{m+1\,n},\ldots,t_{q n})\in(0,\infty)^{q-m}$ such that ${\bf t^1_{n}\bf u_{1n}\in\cN^{\bf\phi|G_{\xi_1}}_{Q_1}}$ and $\bf t^2_{n}\bf u_{2n}\in\cN^{\bf\phi|G_{\xi_2}}_{Q_2}$ for every $n$ sufficiently large, and $t_{in}\to 1$ as $n\to\infty$. Define 
		\begin{equation*}
			\wt u_{in}(x):=\dsum_{[g]\in G/G_{\xi_\nu}}\phi_i(g)t_{hn}u_{in}(g^{-1}x) \qquad\text{if \ } i\in I_h, \text{ \ and \ } h\in Q_\nu.
		\end{equation*}
		Then,  $\wt u_{in}$  is well defined and $\wt u_{in}(gx)=\phi_i(g)\wt u_{in}(x)$ for every $g\in G$, \ $x\in\rn$ \ and \ $i=1,\ldots, \ell$. 
		Since \ ${\supp(\bar u_{hn}\circ g^{-1})\subset B_{nr}(g n\xi_\nu)}$ if $h\in Q_\nu$ and the balls $B_{nr}(g n\xi_\nu)$ with $[g]\in G/G_{\xi_\nu}$ and $\nu=1,2$ are pairwise disjoint, we have that $\wt{\bf u}_n:=(\wt u_{1n},\ldots,\wt u_{\ell n})\in\cN^{\bf\phi}$ and 
		\[
		\dsum_{i\in I_h}\|\wt u_{in}\|^2=|G\xi_\nu|\,|t_{hn}|^2\|\bar u_{hn}\|^2\to|G\xi_\nu|\,\|\bar u_{h}\|^2\qquad\text{if \ }h\in Q_\nu,\ \nu=1,2.
		\]
		Hence,
		\begin{align*}
			c^{\bf\phi} &\leq\dlim_{n\to\infty}\cJ^{\bf\phi}(\wt{\bf u}_n)=\dlim_{n\to\infty}\frac{p-2}{2p}\dsum_{h=1}^q\dsum_{i\in I_h}\|\wt u_{in}\|^2\\
			&=|G\xi_1|\frac{p-2}{2p}\dsum_{h=1}^m\|\bar u_{h}\|^2+|G\xi_2|\frac{p-2}{2p}\dsum_{h=m+1}^q\|\bar u_{h}\|^2\\
			&=|G\xi_1|\cJ^{\bf\phi|G_{\xi_1}}_{Q_1}(\bf u_1)+|G\xi_2|\cJ^{\bf\phi|G_{\xi_2}}_{Q_2}(\bf u_2).
		\end{align*}
		As $\bf u_1\in\cN^{\bf\phi|G_{\xi_1}}_{Q_1}$ and $\bf u_2\in\cN^{\bf\phi|G_{\xi_2}}_{Q_2}$ were arbitrarily chosen, the proof is complete.
	\end{proof}		
	
Consider the \emph{$G$-fixed-point space}
	\[\F(G):=\{x\in\rn:gx=x\text{ for all }g\in G\}.\]	
	
	\begin{proposition} \label{prop:ground_states}
	The following statements hold true:
		\begin{itemize}
			\item[$(i)$] $c^{\bf\phi}\geq\dsum_{h=1}^qc^{\bf\phi}_{h}$.
			\item[$(ii)$]  If $\F(G)\neq 0$, then
			$c^{\bf\phi}=\dsum_{h=1}^qc^{\bf\phi}_{h}$.			
			\item[$(iii)$] If $q\geq 2$ and $c^{\bf\phi} = \dsum_{h=1}^q c_{h}^{\bf\phi}$, then $c^{\bf\phi}$ is not attained.
		\end{itemize}
	\end{proposition}
	
	\begin{proof}
$(i):$ \ Let $\bf u=(\bar u_1,\ldots,\bar u_q)\in\cN^{\bf\phi}$. As $\beta_{ij}<0$ for $(i,j)\in\cK_h$, we have 
			\[0<\|\bar u_h\|^2 \leq \dsum_{(i,j)\in \cI_h}\io\beta_{ij}|u_i|^p|u_j|^p \qquad\text{for every \ }h=1,\ldots,q.\]
			Hence, there exist $(s_1,\ldots,s_q)\in(0,\infty)^q$ such that $s_h\bar u_h\in\cN^{\bf\phi}_{h}$. Setting $\bf s:=(s_1,\ldots,s_q)$ and applying Lemma \ref{lem:N} we obtain
			\[\dsum_{h=1}^q c^{\bf\phi}_{h}\leq \dsum_{h=1}^q \cJ^{\bf\phi}_{h}(s_h\bar u_h)\leq\cJ^{\bf\phi}(\bf s\bf u) \leq \cJ^{\bf\phi}(\bf u).\]
			Hence, $c^{\bf\phi}\geq c^{\bf\phi}_{1}+\cdots+c^{\bf\phi}_{q}$, as claimed.
			
$(ii):$ \ As $\F(G)\neq 0$ we may choose $q$ different points $\xi_1,\ldots,\xi_q\in \F(G)$. Setting $Q_h:=\{h\}$ and iterating Lemma \ref{lem:energy_estimates2} we obtain
			\[
			c^{\bf\phi}\leq \dsum_{h=1}^q c_{h}^{\bf\phi}.
			\]
The equality follows from $(i)$.				
			
$(iii):$ \ We argue by contradiction. Assume that $\bf u=(\bar u_1,\ldots,\bar u_q)\in\cN^{\bf\phi}$ and $\cJ^{\bf\phi}(\bf u)=c^{\bf\phi}$. There are two possibilities. If for every $h$ there exists $k\neq h$ such that
			\[
			\dsum_{(i,j)\in I_h\times I_k}\irn\beta_{ij}| u_i|^{p}|u_j|^{p}\neq 0,
			\]
then $\|\bar u_h\|^2<\dsum_{(i,j)\in \cI_h}\irn\beta_{ij}|u_i|^{p}|u_j|^p$ and, hence, there exists $s_h\in(0,1)$ such that $s_h\bar u_h\in\cN^{\bf\phi}_{h}$ and 
\[
c^{\bf\phi}_{h} \leq \cJ^{\bf\phi}_{h}(s_h\bar u_h)=\dfrac{p-1}{2p}\|s_h\bar u_h\|^2<\dfrac{p-1}{2p}\|\bar u_h\|^2\qquad\text{for every \ } h=1,\ldots,q.
\]
It follows that 
			\[c^{\bf\phi}=\cJ^{\bf\phi}(\bf u)=\dfrac{p-1}{2p}\sum_{h=1}^q\|\bar u_h\|^2>\displaystyle\dsum_{h=1}^q c^{\bf\phi}_{h},\]
contradicting our assumption. On the other hand, if there exists $h$ such that 
			\begin{equation}\label{eq:null}
\dsum_{(s,r)\in I_h\times I_k}\irn \beta_{sr}|u_s|^{p}|u_r|^{p}= 0 \qquad \text{for every \ }k\neq h,
			\end{equation}		
then 
\[\|\bar u_h\|^2=\dsum_{(i,j)\in \cI_h}\irn\beta_{ij}|u_i|^p|u_j|^p\qquad\text{and}\qquad \cJ^{\bf\phi}_{h}(\bar u_h)=c^{\bf\phi}_{h}.\]
Hence, $\bar u_h$ is a solution to the system $(\mathscr S_h^{\bf\phi})$ with $\|\bar u_h\|\neq 0$. So $u_i\neq 0$ for some $i\in I_h$ and it solves the equation
\[-\Delta u_i+ u_i = \dsum_{j\in I_h}\beta_{ij}|u_j|^p|u_i|^{p-2}u_i,\qquad u_i\in H^1(\rn).\]
As $q\geq 2$ there exists $k\neq h$ and $j\in I_k$ such that $u_j\neq 0$, because $\|\bar u_k\|\neq 0$. Since all $\beta_{sr}$ in the sum \eqref{eq:null} are strictly negative, it follows that $|u_i||u_j|= 0$ a.e. in $\rn$. Hence $u_i=0$ in some subset of positive measure of $\rn$, contradicting the unique continuation principle.
\end{proof}	
	
	The following lemma will be used in the proof of Theorem \ref{thm:splitting}.
	
	\begin{lemma}\label{lem:orbits2} 
		For any given sequence $(y_n)$ in $\rn$, there exist a sequence $(\zeta_n)$ in $\rn$ and a closed subgroup $K$ of $G$ such that, up to a subsequence, the following statements hold true:
		
		\begin{itemize}
			\item[$(a)$] $\mathrm{dist}(Gy_n,\zeta_n)\leq C$ for all $n\in\n$ and some positive constant $C$.
			\item[$(b)$] $G_{\zeta_n} = K$ for all $n\in\n$.
			\item[$(c)$] If $|G/K| <\infty$ then $\lim\limits_{n\to\infty}|g\zeta_n-\bar g\zeta_n|=\infty$ for any pair $g,\bar g\in G$ with $\bar g g^{-1}\notin K$.
			\item[$(d)$] If $|G/K|=\infty$ then, for each $n\in\n$, there exists $g_n\in G$ such that $\lim\limits_{k\to\infty}|g_{n_1}\zeta_k-g_{n_2}\zeta_k|=\infty$ for any pair $n_1,n_2\in\n$ with $n_1\neq n_2$.
		\end{itemize}
	\end{lemma}
	
	\begin{proof}
		This follows from \cite[Lemma 3.2]{CC}.
	\end{proof}
	We are finally ready to prove the main result of this section.
	
\begin{theorem} \label{thm:splitting} 
Let $\bf u_n=(u_{1n},\ldots,u_{\ell n})\in\cN^{\bf\phi}$ be such that $\cJ^{\bf\phi}(\bf u_n)\to c^{\bf\phi}$. Then, after passing to a subsequence, there exist $Q\subset\{1,\ldots,q\}$, a least energy block-wise nontrivial solution $\bf w$ to the subsystem
\begin{equation} \label{eq:subsystem}
\tag{$\mathscr S_{Q}^{\bf\phi}$}\qquad
\begin{cases}
-\Delta w_i+ w_i = \dsum_{j\in I_Q} \beta_{ij}|w_j|^p|w_i|^{p-2}w_i, \\
w_i\in H^1(\rn)^{\phi_i},\qquad i\in I_Q,
\end{cases}
\end{equation} 
whenever $Q\neq\emptyset$, and, for each $k\notin Q$, a sequence $(\xi_{kn})$ in $\rn$, a closed subgroup $G_k$ of $G$, and a least energy nontrivial solution $\bar v_k$ to the cooperative subsystem
\begin{equation} \label{eq:cooperative}
\tag{$\mathscr S_{k}^{\bf\phi|G_k}$}\qquad
\begin{cases}
-\Delta v_i+ v_i = \dsum_{j\in I_k} \beta_{ij}|v_j|^p|v_i|^{p-2}v_i, \\
v_i\in H^1(\rn)^{\phi_i|G_k},\qquad i\in I_k,
\end{cases}
\end{equation}
with the following properties:
\begin{itemize}
\item[$(i)$] If $i\in I_Q$, then $\dlim_{n \to \infty}\|u_{in}-w_i\|=0$.
\item[$(ii)$] If $k\notin Q$, then $G_{\xi_{kn}}=G_k$ for all $n\in\n$, $|G/G_k|<\infty$, $\dlim\limits_{n\to\infty}|g\xi_{kn}-\bar g\xi_{kn}|=\infty$ for any $g,\bar g\in G$ with $\bar g g^{-1}\notin G_k$, and $\dlim_{n \to \infty}\mathrm{dist}(G\xi_{kn},G\xi_{k'n})=\infty$ whenever $k'\neq k$, $k'\notin Q$.
\item[$(iii)$] If $Q\neq\emptyset$ and $k\notin Q$, then $\dlim_{n\to \infty}|\xi_{kn}|=\infty$.
\item[$(iv)$] If $i\in I_k$ and $k\notin Q$, then
\[\dlim_{n\to\infty}\left\|u_{in}-\dsum_{[g]\in G/G_k}\phi_i(g)(v_i\circ g^{-1})( \ \cdot \ -g\xi_{kn})\right\|=0.\]
\item[$(v)$] $c^{\bf\phi} = c_{Q}^{\bf\phi} + \dsum_{k\notin Q}|G/G_k|\,c^{\bf\phi|G_k}_{k}$, \ where $c_{Q}^{\bf\phi}:=0$ if $Q=\emptyset$.
\end{itemize}
Moreover, if $u_{in}\geq 0$ for all $n\in\n$, then $w_i\geq 0$ if $i\in I_Q$ and $v_i\geq 0$ if $i\notin I_Q$.
	\end{theorem}
	
	\begin{proof}
		We split the proof into three steps.
		\medskip
		
		\textsc{Step 1.} \ \emph{The choice of the concentration points.}
		
Since $\bf u_n\in\cN^{\bf\phi}$ and $\beta_{ij}<0$ for $(i,j)\in \cK_h$ by assumption $(B_1)$, Lemma \ref{lem:N} yields a constant $d_0>0$ such that
		\begin{equation*}
			\dsum_{(i,j)\in \cI_h}\irn\beta_{ij}|u_{in}|^p|u_{jn}|^p>d_0\qquad\text{for all \ }n\in\n \text{ \ and \ } h=1,\ldots,q.
		\end{equation*}
Hence, there exist $i \in I_h$ and $c_0>0$ such that
\[\irn\beta_{ii}|u_{in}|^{2p}>c_0\qquad\text{for all \ }n\in\n.\]
Invoking \cite[Lemma 2.4]{cp} and Ekeland's variational principle \cite[Theorem 8.5]{W} we may assume without loss of generality that $(\cJ^{\bf\phi})'(\bf u_n)\to 0$ in $(\cH^{\bf\phi})'$. By Lions' lemma \cite[Lemma 1.21]{W} there exists $\theta>0$ such that, after passing to a subsequence,
		\begin{equation*}
\sup_{y \in \rn}\dsum_{(i,j)\in \cI_h}\int_{B_1(y)}\beta_{ij}|u_{in}|^p|u_{jn}|^p\geq\sup_{y \in \rn}\int_{B_1(y)}\beta_{ii}|u_{in}|^{2p}>2\theta,
		\end{equation*}
and we may choose $y_{hn}\in\rn$ satisfying
		\begin{equation*}
\dsum_{(i,j)\in \cI_h}	\int_{B_1(y_{hn})}\beta_{ij}|u_{in}|^p|u_{jn}|^p> \theta \qquad\text{for all \ }n\in\n \text{ \ and \ } h=1,\ldots,q.
		\end{equation*}
Applying Lemma \ref{lem:orbits2} to $(y_{hn})$ we obtain a sequence $(\zeta_{hn})$, a closed subgroup $K_h = G_{\zeta_{hn}}$ and a constant $C>0$ such that $\mathrm{dist}(G y_{hn},\zeta_{hn})\leq C$ for all $n\in \n$ and $h=1,\ldots,q$. Since $|u_{in}|$ is $G$-invariant, taking $g_{hn}\in G$ such that \ $|g_{hn}y_{hn}-\zeta_{hn}|=\mathrm{dist}(Gy_{hn},\zeta_{hn})$ we get
		\begin{equation} \label{eq:nonnull_2}
			\dsum_{(i,j)\in \cI_h}	\int_{B_{C+1}(\zeta_{hn})}\beta_{ij}|u_{in}|^p|u_{jn}|^p\geq  \dsum_{(i,j)\in \cI_h}	\int_{B_1(g_{hn}y_{hn})}\beta_{ij}|u_{in}|^p|u_{jn}|^p> \theta,
		\end{equation}
for all $n\in\n$ and $h=1,\ldots,q$. \ We claim that $|G/K_h|<\infty$. Otherwise, by Lemma \ref{lem:orbits2}$(d)$, for each $M\in\n$ we would be able to choose $g_{h1},\ldots,g_{hM}\in G$ such that
		\[|g_{hm} \zeta_{hn} - g_{hm'} \zeta_{hn}| \geq 2(C+1) \quad\text{for all pairs \ }m\neq m', \ \ m,m' = 1,\ldots,M,\]
and large $n \in \n$. As a consequence, $B_{C+1}(g_{hm}\zeta_{hn})\cap B_{C+1}(g_{hm'}\zeta_{hn})=\emptyset$. Thus,
		\begin{align*}
			\dsum_{(i,j)\in \cI_h}\irn\beta_{ij}|u_{in}|^p|u_{jn}|^p&\geq\dsum_{m=1}^M \dsum_{(i,j)\in \cI_h}	\int_{B_{C+1}(g_{h\bar m}\zeta_{hn})}\beta_{ij}|u_{in}|^p|u_{jn}|^p\\
			&=M \dsum_{(i,j)\in \cI_h}	\int_{B_{C+1}(\zeta_{hn})}\beta_{ij}|u_{in}|^p|u_{jn}|^p> M\theta,
		\end{align*}
which is a contradiction, because $M$ is arbitrary and, as $(\bf u_n)$ is bounded in $\cH^{\bf\phi}$, $(u_{in})$ is bounded in $L^{2p}(\rn)$. This shows that $|G/K_h|<\infty$. Therefore, by Lemma \ref{lem:orbits2}$(c)$,
		\begin{equation} \label{eq:orbits}
			\lim_{n\to\infty}|g\zeta_{hn}-\bar g\zeta_{hn}|=\infty\quad\text{for all \ }g,\bar g\in G\text{ \ with \ }\bar g g^{-1}\notin K_h,
		\end{equation}
and every $h=1,\ldots,q$.		
		
Let $Q_1,\ldots,Q_t$ be nonempty and pairwise disjoint subsets of $\{1,\ldots,q\}$ such that $Q_1\cup\cdots\cup Q_t=\{1,\ldots,q\}$ and, after passing to a subsequence,
		\begin{align*}
			&(\mathrm{dist}(G\zeta_{hn},G\zeta_{kn}))\text{ is bounded} &&\text{if \ }h,k\in Q_r\text{ for some } r, \\
			&\mathrm{dist}(G\zeta_{hn},G\zeta_{kn})\to\infty \ \ \text{as} \ \ n\to \infty &&\text{if \ }h\in Q_r\text{ and }k\in Q_{r'}\text{ with }r\neq r'.
		\end{align*}
Choose $h_r\in Q_r$ and set $\xi_{rn}:=\zeta_{h_rn}$. Fix $R>\mathrm{dist}(\xi_{rn},G\zeta_{hn})+C+1$ for every $n\in\n$ and $h\in Q_r$. Then
		\begin{align*}
\mathrm{dist}(G\xi_{rn},G\xi_{r'n})\to\infty \text{ \ as \ }n\to\infty \qquad\text{if \ }r\neq r',
		\end{align*}
and from \eqref{eq:nonnull_2} we obtain
		\begin{equation} \label{eq:nonull2_2}
			\sum_{(i,j)\in \cI_h}\int_{B_R(\xi_{rn})}\beta_{ij}|u_{in}|^p|u_{jn}|^p > \theta>0\quad\text{for all \ }n\in\n, \quad h\in Q_r.
		\end{equation}
Finally, set
		\[
		\Gamma_r:=G_{\xi_{rn}}=K_{h_{r}}\qquad r=1,\ldots,t.
		\]
		Note that if $\F(\Gamma_r) = \{0\}$, then $\xi_{rn} = 0$ for all $n\in \n$ and, hence, $\F(\Gamma_{r'})\neq \{0\}$ for each $r' \neq r.$
		
\medskip		
		
		\textsc{Step 2.} \ \emph{Finding a block-wise nontrivial solution to a subsystem.}
		
		Define
		\[w_{in}(x):=u_{in}(x+\xi_{rn})\quad\text{if \ }i\in I_{Q_r}.\]
Observe that  $w_{in} \in H^1(\rn)^{\phi_i|\Gamma_r}$ if $i\in I_{Q_r}$. Since the sequence $(w_{in})$ is bounded in $H^1(\rn)$, passing to a subsequence, we have that $w_{in}\rh w_i$ weakly in $H^1(\rn)^{\phi_i|\Gamma_r}$, $w_{in}\to w_i$ in $L^{2p}_\mathrm{loc}(\rn)$ and $w_{in}\to w_i$ a.e. in $\rn$. Hence, $w_{in}\geq 0$ if $u_{in}\geq 0$ for all $n\in\n$. As, by \eqref{eq:nonull2_2},
		\begin{equation*}
\dsum_{(i,j)\in \cI_h}	\int_{B_R(0)}\beta_{ij}|w_{in}|^p|w_{jn}|^p=\dsum_{(i,j)\in \cI_h}\int_{B_R(\xi_{rn})}\beta_{ij}|u_{in}|^p|u_{jn}|^p > \theta>0\quad\text{for all \ }n\in\n, \quad h\in Q_r,
		\end{equation*}
we deduce that $w_i\neq 0$ for some $i\in I_h$. So the vector $\overline w_h=(w_{\ell_{h-1}+1},\ldots,w_{\ell_h})$ is nonzero. 

Fix $i\in I_{Q_r}$ and let $g_1,\ldots,g_m\in G$ represent all the different cosets in $G/\Gamma_r$. By  \eqref{eq:orbits}, $|g_j\xi_{rn}-g_s\xi_{rn}|\to\infty$ if $j\neq s$. Therefore, 
\[\phi_i(g_j)\,(w_{in}\circ g_j^{-1})-\sum_{s=j+1}^{m}\phi_i(g_s)\,(w_{i}\circ g_s^{-1})(\,\cdot\,-g_s\xi_{rn}+g_j\xi_{rn})\rh\phi_i(g_j)\,(w_i\circ g_j)\]
weakly in $H^1(\rn)$. It follows that
\begin{align*}
&\Big\|\phi_i(g_j)\,(w_{in}\circ g_j^{-1})-\sum_{s=j+1}^{m}\phi_i(g_s)\,(w_{i}\circ g_s^{-1})(\,\cdot\,-g_s\xi_{rn}+g_j\xi_{rn})\Big\|^2\\
&=\Big\|\phi_i(g_j)\,(w_{in}\circ g_j^{-1})-\sum_{s=j}^{m}\phi_i(g_s)\,(w_{i}\circ g_s^{-1})(\,\cdot\,-g_s\xi_{rn}+g_j\xi_{rn})\Big\|^2+\Big\|\phi_i(g_j)\,(w_i\circ g_j)\Big\|^2+o_n(1).
\end{align*}
Since $u_{in}$ is $\phi_i$-equivariant, the change of variable $y=z-g_j\xi_{rn}$ yields
\begin{align*}
\Big\|u_{in}-\sum_{s=j+1}^{m}\phi_i(g_s)\,(w_{i}\circ g_s^{-1})(\,\cdot\,-g_s\xi_{rn})\Big\|^2=\Big\|u_{in}-\sum_{s=j}^{m}\phi_i(g_s)\,(w_{i}\circ g_s^{-1})(\,\cdot\,-g_s\xi_{rn})\Big\|^2+\|w_{i}\|^2+o_n(1),
\end{align*}
and iterating this identity we obtain
\begin{equation} \label{eq:split_components}
\|u_{in}\|^2-\Big\|u_{in}-\sum_{s=1}^{m}\phi_i(g_s)\,(w_{i}\circ g_s^{-1})(\,\cdot\,-g_s\xi_{rn})\Big\|^2+m\|w_{i}\|^2+o_n(1)\quad\text{for each \ }i\in I_{Q_r},
\end{equation}
with $m=|G/\Gamma_r|$.
		
Next, we fix $r\in\{1,\ldots,t\}$ and define
		\begin{equation*}
u^r_{in}(x):=u_{in}(x+\xi_{rn}).
		\end{equation*}
Then, $u^r_{in} \in H^1(\rn)^{\phi_i|\Gamma_r}$ and, after passing to a subsequence, $u^r_{in}\rh u^r_{i}$ weakly in $H^1(\rn)$. Note that $u^r_{i}=w_i$ if $i\in I_{Q_r}$.

For any $\vp\in\cC_c^\infty(\rn)^{\phi_i|\Gamma_r}$ define $\vp_{n}(y):=\vp(y-\xi_{rn})$ and set
\[\what\vp_n(y):=\sum_{[g]\in G/\Gamma_r}\phi_i(g)\vp_n(gy).\]
Then $\what\vp_n\in H^1(\rn)^{\phi_i}$. As $\bf u_n\in\cH^{\bf \phi}$ and $(\cJ^{\bf \phi})'(\bf u_n)\to 0$ in $(\cH^{\bf \phi})'$, we deduce that
\begin{align*}
0&=\lim_{n\to\infty}\partial_i\cJ^{\bf\phi}(\bf u_n)\what\vp_n=\lim_{n\to\infty}\sum_{[g]\in G/\Gamma_r}\partial_i\cJ^{\bf\phi}(\bf u_n)[\phi_i(g)(\vp_n\circ g)]\\
&=\lim_{n\to\infty}|G/\Gamma_r|\partial_i\cJ^{{\bf\phi}|\Gamma_r}(\bf u_n^r)\varphi=|G/\Gamma_r|\partial_i\cJ^{{\bf\phi}|\Gamma_r}(\bf u^r){\varphi}\quad\text{for all \ }i=1,\ldots,\ell,
\end{align*}
where $\bf u^r=(\overline u_1^r,\ldots,\overline u_q^r)$ with $\overline u_h^r=(u^r_{\ell_{h-1}+1},\ldots,u^r_{\ell_h})$.  So $\bar u_h^r=\bar w_h\neq 0$ if $h\in {Q_r}$. The above identity shows that $\bf u^r$ is a critical point of $\cJ^{\bf \phi}:\cH^{\bf \phi}\to \r$, so its nontrivial block-wise components $\overline u_h^r$ provide a nontrivial block-wise solution $\bf w^r$ to the system
\begin{equation} \label{eq:subsystem_2}
	\tag{$\mathscr S^{\bf\phi|\Gamma_{r}}_{\what Q_r}$}
	\begin{cases}
		-\Delta u_i+u_i=\dsum_{j\in I_{\what Q_r}}\beta_{ij}|u_j|^{p}|u_i|^{p-2}u_i,\\
		u_i \in H^1(\rn)^{\phi_i|\Gamma_r}, \qquad i\in I_{\what Q_r},
	\end{cases}
\end{equation}
where $\what Q_r :=\{h:1\leq h\leq q, \ \bar u_h^r\neq 0\}\supset Q_r$. 
		\medskip
		
		\textsc{Step 3.} \ \emph{The conclusion.}
		
		We distinguish two cases.
\begin{itemize}
\item[$(I):$] Assume $\F(\Gamma_r)\neq\{0\}$ for every $r=1,\ldots,t$. 

Since $\bf w^r$ is a block-wise nontrivial solution to $(\mathscr{S}^{\bf \phi|\Gamma_r}_{\what Q_r})$, multiplying the $i$-th equation by $w_i$ for $i\in I_h$ and $h\in Q_r$, and integrating we get
\begin{equation} \label{eq:t_h}
	\| w_i\|^2 = \sum_{j\in I_{\what Q_r}}\irn\beta_{ij}|w_i|^p|u^h_{j}|^p\leq \dsum_{j\in I_h}\irn\beta_{ij}|w_i|^p|w_j|^p.
\end{equation}
So there exists $t_h \in (0,1]$ such that $\|t_h\bar w_h\|^2 = \dsum_{i,j\in I_h}\irn \beta_{ij}|t_hw_i|^p|t_hw_j|^p$. Then, $t_h\bar w_h \in \cN_h^{\bf \phi|\Gamma_r}$ and
\begin{equation} \label{eq:barwh}
c^{\bf \phi|\Gamma_r}_{h}\leq\dfrac{p-1}{2p}\|t_h\bar w_h\|^2\leq \dfrac{p-1}{2p}\|\bar w_h\|^2.
\end{equation} 
Since $\F(\Gamma_r)\neq\{0\}$ for every $r$, Proposition \ref{prop:ground_states}$(ii)$ yields $c_{Q_r}^{\bf \phi| \Gamma_r} = \dsum_{h\in Q_r}c_h^{\bf \phi|\Gamma_r}$.  Applying Lemma \ref{lem:energy_estimates2} recursively and using \eqref{eq:split_components} and \eqref{eq:barwh} we obtain
\begin{align} \label{eq:totalenergy}
&\sum_{r=1}^{t}\sum_{h\in Q_r}|G/\Gamma_r|c_h^{\bf \phi|\Gamma_r}=\sum_{r=1}^{t}|G/\Gamma_r|\,c^{{\bf\phi}|\Gamma_r}_{Q_r}\geq c^{\bf\phi} =\lim_{n\to\infty}\dfrac{p-1}{2p}\sum_{r=1}^{t}\sum_{h\in Q_r}\|\bar u_{hn}\|^2\\
&\qquad\geq \sum_{r=1}^{t}\sum_{h\in Q_r}\dfrac{p-1}{2p}|G/\Gamma_r|\,\|\bar w_{h}\|^2\geq \sum_{r=1}^{t}\sum_{h\in Q_r}\dfrac{p-1}{2p}|G/\Gamma_r|\,\|t_h\bar w_{h}\|^2\geq\sum_{r=1}^{t}\sum_{h\in Q_r}|G/\Gamma_r|c_h^{\bf \phi|\Gamma_r}. \nonumber
\end{align}
Therefore, $t_h=1$ and from \eqref{eq:t_h} we get that
\[\irn\beta_{ij}|w_i|^p|u^h_{j}|^p=0\qquad\text{for every \ }j\in I_{\what Q_r}\smallsetminus I_h.\]
Choosing $i\in I_h$ such that $w_i\neq 0$, as $\bf w^r$ solves the system, we have that $w_i(x)\neq 0$ a.e. in $\rn$. As a consequence $u^h_{j}=0$ for every $j\in I_{\what Q_r}\smallsetminus I_h$. This shows that $\what Q_r=Q_r$, and from \eqref{eq:barwh} and \eqref{eq:totalenergy} we get
\[c_{Q_r}^{\bf \phi| \Gamma_r}=\sum_{h\in Q_r}c^{\bf \phi|\Gamma_r}_{h}=\dfrac{p-1}{2p}\sum_{h\in Q_r}\|\bar w_h\|^2=\dfrac{p-1}{2p}\|\bf w^r\|^2.\]

Therefore $c_{Q_r}^{\bf \phi| \Gamma_r}$ is attained. As $\F(\Gamma_r)\neq\{0\}$, Proposition \ref{prop:ground_states} implies that each $Q_r$ must consist of a single element. We relabel these sets as $Q_k=\{k\}$ and the corresponding points $\xi_{rn}$ as $\xi_{kn}$ for each $k\in\{1,\ldots,q\}$. From the way the sets $Q_r$ were chosen, we get $\mathrm{dist}(G\xi_{hn},G\xi_{kn})\to\infty$ if $h\neq k$. Setting $Q:=\emptyset$, $G_k:=\Gamma_{k}=G_{\xi_{kn}}$ and $\bar v_k:=\bar w_k$, we have that $\bar v_k$ is a least energy nontrivial solution to the cooperative system \eqref{eq:cooperative} and we derive from \eqref{eq:orbits} and \eqref{eq:totalenergy} that these data satisfy $(ii)$ and $(v)$. Finally, from \eqref{eq:split_components} and \eqref{eq:totalenergy} we obtain
\begin{equation*}
\lim_{n\to\infty}\Big\|u_{in}-\dsum_{g\in G/G_h}\phi_i(g)(w_i\circ g^{-1})( \ \cdot \ - g\xi_{hn})\Big\|^2=0\quad \text{if \ }i\in I_h,
\end{equation*}
as claimed in $(iv)$.

\item[$(II):$] Assume $\F(\Gamma_{r_0})=\{0\}$ for some $r_0\in\{1,\ldots,t\}$. To simplify notation, let $r_0=1$.

Then $\xi_{1n}=0$ for every $n\in\n$, $\Gamma_1=G$ and $\F(\Gamma_r)\neq \{0\}$ for all $r\neq 1$. Setting $\overline Q_r := Q_r \smallsetminus \what Q_1$, choosing $t_h\in(0,1]$ as before if $h\in\overline Q_r$, and arguing as in $(I)$ we derive
		\begin{align} \label{eq:totalenergy3}
&\sum_{r=2}^t\sum_{h\in \overline Q_r}|G/\Gamma_r|c_h^{\bf \phi|\Gamma_r} + c_{\what Q_1}^{\bf \phi}=
\sum_{r=2}^t|G/\Gamma_r|\,c^{{\bf\phi}|\Gamma_r}_{\overline Q_r} + c_{\what Q_{1}}^{\bf \phi}\geq c^{\bf\phi} \\
& =\dfrac{p-1}{2p}\Big(\lim_{n\to\infty}\sum_{r=2}^t\sum_{h\in \overline Q_r}\|\overline u_{hn}\|^2 + \sum_{h\in\what Q_1}\|\overline u_{hn}\|^2\Big)\geq \dfrac{p-1}{2p}\Big(\sum_{r=2}^t\sum_{h\in \overline Q_r}	|G/\Gamma_r|\,\|\bar w_{h}\|^2 + \sum_{h\in\what Q_1}\|\bar u_h^1\|^2 \Big)\nonumber\\
&\geq\dfrac{p-1}{2p}\Big(\sum_{r=2}^t\sum_{h\in \overline Q_r}|G/\Gamma_r|\,\|t_h\bar w_{h}\|^2 + \|\bf w^1\|^2 \Big)\geq \sum_{r=2}^t\sum_{h\in \overline Q_r}|G/\Gamma_r|c_h^{\bf \phi|\Gamma_r} + c_{\what Q_1}^{\bf \phi}.\nonumber
	\end{align}
It follows that $\bar u_{hn}\to\bar u^1_h\neq 0$ strongly in $H^1(\rn)$ for all $h\in \what Q_1$. Since  $|\xi_{rn}|\to\infty$ for every  $r\neq 1$ and the weak limit of $\bar u_{hn}(\,\cdot\,+\xi_{rn})$ is nonzero for every $h\in Q_r$, we must have  $\what Q_1=Q_1$. 
Set $Q:=Q_1$ and $\bf w:=\bf w^1$. As in case $(I)$ we see that $Q_r$ must consist of a single element if $r\neq 1$. We relabel these sets as $Q_k=\{k\}$ and set $\xi_{kn}:=\zeta_{kn}$, $G_k:=\Gamma_{k}=G_{\xi_{kn}}$ and $\bar v_k:=\bar w_k$ for each $k\notin Q$. These data have the properties stated in Theorem \ref{thm:splitting}.
\end{itemize}	
The proof is complete.
\end{proof}
	
\begin{corollary} \label{cor:main_inequality}
Under the assumptions of \emph{Theorem \ref{thm:splitting}}, the following statements hold true:
\begin{itemize}
\item[$(i)$] If $q\geq 2$ and $\F(G)\neq 0$ the system $(\mathscr S^{\bf\phi})$ does not have a least energy block-wise nontrivial solution.
\item[$(ii)$] If for every $h\in\widehat Q:=\{1,\ldots,q\}$ one has that
\begin{equation}  \label{eq:condition}
c^{\bf\phi}<c_{\widehat Q\smallsetminus\{h\}}^{\bf\phi} + |G\xi|\,c_h^{{\bf\phi}|G_\xi}\qquad\text{for every \ }\xi\in\rn\smallsetminus\{0\},
\end{equation}
then the system $(\mathscr S^{\bf\phi})$ has a least energy block-wise nontrivial solution.
\end{itemize}
\end{corollary}	

\begin{proof}
$(i)$ is an immediate consequence of Proposition \ref{prop:ground_states}.

$(ii):$ \ Let $Q\subset\widehat Q$ be as in Theorem \ref{thm:splitting}. If $Q=\widehat Q$, the function $\bf w$ given by that theorem is a least energy block-wise nontrivial solution to $(\mathscr S^{\bf\phi})$. 
If $Q\neq\widehat Q$ then, for each $k\notin Q$, there are $(\xi_{kn})$, $G_k$ and $\bar v_k$ with the properties stated in Theorem \ref{thm:splitting}. So we may choose $n_0$ large enough so that $\xi_k:=\xi_{kn_0}$ satisfies $G\xi_k\neq G\xi_{k'}$ if $k,k'\notin Q$ and $k\neq k'$ and, furthermore, $G\xi_k\neq\{0\}$ if $Q\neq\emptyset$ for every $k\notin Q$. Then, as $G_k=G_{\xi_k}$, iterating Lemma \ref{lem:energy_estimates2} and using the identity in Theorem \ref{thm:splitting}$(v)$, we obtain
\[c^{\bf\phi} = c_{Q}^{\bf\phi} + \dsum_{k\notin Q}|G/G_k|\,c^{{\bf\phi}|G_k}_{k}\geq c_{\widehat Q\smallsetminus\{h\}}^{\bf\phi} + |G/G_h|\,c_h^{{\bf\phi}|G_h},\]
where $h\notin Q$ is chosen so that $\xi_h\neq 0$ if $q\geq 2$. This inequality stands in contradiction with \eqref{eq:condition} unless $\xi_h=0$, in which case $q=1$, $Q=\emptyset,$  $G_h=G$, the system $(\mathscr S^{{\bf\phi}|G_h}_h)$ is $(\mathscr S^{\bf\phi})$, and $\bar v_h$ is a least energy block-wise nontrivial solution to $(\mathscr S^{\bf\phi})$. This completes the proof.
\end{proof}

\begin{remark}
\emph{If the system is cooperative, i.e., if $q=1$, condition \eqref{eq:condition} implies that $\F(G)=\{0\}$. This last condition is actually not necessary, see \cite[Theorem 1.6]{cp} and Section \ref{sec:cooperative} below.}
\end{remark}

\begin{corollary} \label{cor:infinite_orbits}
If $|G\xi|=\infty$ for every $\xi\in\rn\smallsetminus\{0\}$, the system $(\mathscr S^{\bf\phi})$ has a least energy block-wise nontrivial solution.
\end{corollary}

\begin{proof}
This is an immediate consequence of Corollary \ref{cor:main_inequality}\,$(ii)$.
\end{proof}

The following example shows that there are groups with infinite orbits admitting a surjective homomorphism onto $\z_2$ satisfying assumption $(A_\phi)$.

\begin{example} \label{example2}
Let $G_\infty$ be the group generated by $\{\mathrm{e}^{\mathrm{i}\vartheta}:\vartheta\in[0,2\pi)\}\cup\{\tau\}\cup O(N-4)$, acting on a point $(z_1,z_2,y)\in\cc\times\cc\times\r^{N-4}\equiv\rn$ as follows:
	\begin{align*}
		\mathrm{e}^{\mathrm{i}\vartheta}(z_1,z_2,y)&:=(\mathrm{e}^{\mathrm{i}\vartheta}z_1,\mathrm{e}^{\mathrm{i}\vartheta}z_2,y), \qquad\qquad\tau(z_{1},z_{2},y):=(z_{2},z_{1},y),\\
		\alpha(z_1,z_2,y)&:=(z_1,z_2,\alpha y)\quad\text{if \ }\alpha\in O(N-4).
	\end{align*}
and let $\t:G_\infty \to\z_2$ be the homomorphism satisfying $\t(\mathrm{e}^{\mathrm{i}\vartheta})=1$, $\t(\tau)=-1$, and $\t(\alpha)=1$ for every $\alpha\in O(N-4)$. Then $\t$ satisfies $(A_{\t})$, and $|G_\infty\xi|=\infty$ for every $\xi\in\rn\smallsetminus\{0\}$ if $N=4$ or $N\geq 6$. 
\end{example}

\begin{theorem} \label{thm:existence}
Let $N=4$ or $N\geq 6$, and $G:=G_\infty$ and $\t$ be as in \emph{Example} \ref{example2}. Let $\{1,\ldots,q\}=Q^+\cup Q^-$ with $Q^+\cap Q^-=\emptyset$ and set $\phi_h\equiv 1$ if $h\in Q^+$ and $\phi_h=\t$ if $h\in Q^-$. Then there exists a least energy block-wise nontrivial solution $\bf w=(\bar w_1,\ldots,\bar w_q)$ to the system $(\mathscr S^{\bf\phi})$ such that every nontrivial component of $\bar w_h$ is positive if $h\in Q^+$ and every nontrivial component of $\bar w_h$ changes sign if $h\in Q^-$.
\end{theorem}

\begin{proof}
The existence of a least energy block-wise nontrivial solution to $(\mathscr S^{\bf\phi})$ is given by Corollary \ref{cor:infinite_orbits}. By Remark \ref{rem:positive} we may choose the nontrivial components $w_i$ to be positive if $i\in I_{Q^+}$, and every nontrivial $w_i$ with $i\in I_{Q^-}$ changes sign because $\phi_h$ is surjective for $h\in Q^-$.
\end{proof}

\section {Fully nontrivial solutions} 
\label{sec:fully_nontrivial}

Fix a closed subgroup $G$ of $O(N)$ and, for each $h=1,\ldots,q$, fix a continuous homomorphism $\phi_h:G\to\z_2$ satisfying $(A_{\phi_h})$. Set $\bf\phi:=(\phi_1,\ldots,\phi_q)$.
Our aim now is to give conditions guaranteeing that every least energy block-wise nontrivial solution to the system $(\mathscr S^{\bf\phi})$ is fully nontrivial. Recall we are always assuming $N\geq 4$, so $1<p<2$.

\begin{lemma} \label{lem:cp}
Assume $(B_1)$. Let $\bf u=(u_1,\ldots,u_\ell)\in\cN^{\bf\phi}$ be such that some component $u_i=0$ and $i\in I_h$. Given $\vp\in H^1(\rn)^{\phi_h}$ and $\eps>0$, define $\bf u_\eps=(u_{\eps 1},\ldots,u_{\eps \ell})$ by $u_{\eps i}:=\eps\vp$ and $u_{\eps j}:=u_j$ if $j\neq i$. Then, there exist $\eps_0>0$ and a $\cC^1$-map $\bf t:(-\eps_0,\eps_0)\to (0,\infty)^q$ satisfying $\bf t(\eps)\bf u_\eps\in\cN^{\bf\phi}$, \ $\bf t(0)=(1,\ldots,1)$, \ $\bf t'(0)=(0,\ldots,0)$ \ and
$$\cJ^{\bf\phi}(\bf t(\eps)\bf u_\eps)-\cJ^{\bf\phi}(\bf u)=\eps^p\Big(-\frac{1}{p}\sum_{\substack{j=1 \\ j\neq i}}^\ell\beta_{ij}\irn|\vp|^p|u_j|^p + o(1)\Big).$$
\end{lemma}

\begin{proof}
See \cite[Lemmas 3.1 and 3.3]{cp}.
\end{proof}

Let $d_{\bf\phi}>0$ be as in Lemma \ref{lem:N} and set
$$S^{\bf\phi}:=\min_{h=1,\ldots,q}\,\inf_{\substack{v\in H^1(\rn)^{\phi_h} \\ v\neq 0}}\frac{\|v\|^2}{|v|_{2p}^2},$$
where $|v|_{2p}$ denotes the norm of $v$ in $L^{2p}(\rn)$. Define
\begin{equation}\label{eq:C*}
C_{\bf\phi}:=\left(\frac{pd_{\bf\phi}}{(p-1)(S^{\bf\phi})^\frac{p}{p-1}}\right)^p.
\end{equation}
We stress that $C_{\bf\phi}$ depends on the chosen symmetries. We shall consider the following assumption.
\begin{itemize}
\item[$(B_3^{\bf\phi})$] If $q\geq 2$ then, for every $h\in\{1,\ldots,q\}$ such that $\ell_{h}-\ell_{h-1}\geq 2$, the following inequality holds true:
$$\Big(\min_{\{i,j\}\in E_h}\beta_{ij}\Big)\left[\frac{\dmin_{h=1,\ldots,q}\dmax_{i\in I_h}\beta_{ii}}{\dsum_{(i,j)\in\cI_h}\beta_{ij}}\right]^\frac{p}{p-1}>C_{\bf\phi}\sum_{(i,j)\in\cK_h}|\beta_{ij}|.$$
\end{itemize}

If $i\in I_h$ we denote by $N_h(i):=\{j\in I_h:j\neq i, \ \beta_{ij}>0\}$ the neighborhood of the vertex $i$ in the graph whose set of vertices is $I_h$ and whose set of edges is $E_h:=\{\{i,j\}:i,j\in I_h, \ i\neq j, \ \beta_{ij}>0\}$.

\begin{lemma} \label{lem:neighbors}
Assume $(B_1)$ and $(B_3^{\bf\phi})$. Let $\bf u=(u_1,\ldots,u_\ell)\in\cN^{\bf\phi}$ be such that $\cJ^{\bf\phi}(\bf u)=c^{\bf\phi}$. If $u_{j_0}\neq 0$ and $j_0\in I_h$, then $u_i\neq 0$ for every $i\in N_h(j_0)$.
\end{lemma}

\begin{proof}
Arguing by contradiction, assume that $u_{i_0}=0$ for some $i_0\in N_h(j_0)$. Choose $i_0^*\in N_h(i_0)$ such that $|u_{i_0^*}|_{2p}\geq|u_j|_{2p}$ for every $j\in N_h(i_0)$. As $j_0\in N_h(i_0)$ and $u_{j_0}\neq 0$, we have that $|u_{i_0^*}|_{2p}>0$. Furthermore,
$$S^{\bf\phi}|u_{i_0^*}|_{2p}^2\leq \|u_{i_0^*}\|^2\leq \sum_{(i,j)\in\cI_h}\beta_{ij} \irn|u_i|^p|u_j|^p\leq \sum_{(i,j)\in\cI_h}\beta_{ij} |u_i|_{2p}^p|u_j|_{2p}^p\leq\sum_{(i,j)\in\cI_h}\beta_{ij}|u_{i_0^*}|_{2p}^{2p}.$$
Therefore,
\begin{align}\label{eq:nontrivial1}
\sum_{j\in N_h(i_0)}\beta_{i_0j}\irn|u_{i_0^*}|^p|u_j|^p &\geq \beta_{i_0i_0^*}|u_{i_0^*}|_{2p}^{2p}\geq\Big(\min_{\{i,j\}\in E_h}\beta_{ij}\Big)|u_{i_0^*}|_{2p}^{2p} \\
&\geq \Big(\min_{\{i,j\}\in E_h}\beta_{ij}\Big)\left[\frac{S^{\bf\phi}}{\dsum_{(i,j)\in\cI_h}\beta_{ij}}\right]^\frac{p}{p-1}.\nonumber
\end{align}
On the other hand, as $\cJ^{\bf\phi}(\bf u)=\frac{p-1}{2p}\|\bf u\|^2=c^{\bf\phi}$, from Lemma \ref{lem:N} we derive
\begin{align}\label{eq:nontrivial2}
\Big(\irn|u_{i_0^*}|^p|u_j|^p\Big)^\frac{2}{p}&\leq |u_{i_0^*}|_{2p}^2|u_j|_{2p}^2\leq(S^{\bf\phi})^{-2}\|u_{i_0^*}\|^2\|u_j\|^2\leq \left((S^{\bf\phi})^{-1}\|\bf u\|^2\right)^2 \\
&\leq\left(\frac{pd_{\bf\phi}}{(p-1)S^{\bf\phi}}\Big(\min_{h=1,\ldots,q}\max_{i\in I_h}\beta_{ii}\Big)^{-\frac{1}{p-1}}\right)^2.\nonumber
\end{align}
Set $\vp:=u_{i_0^*}$. If $q=1$ then
$$\sum_{\substack{j=1 \\ j\neq i_0}}^\ell\beta_{i_0j}\irn|\vp|^p|u_j|^p =\sum_{j\in N_h(i_0)}\beta_{i_0j}\irn|u_{i_0^*}|^p|u_j|^p>0.$$
If $q\geq 2$, inequalities \eqref{eq:nontrivial1} and \eqref{eq:nontrivial2} and assumption $(B_3^{\bf\phi})$ yield
\begin{align*}
&\sum_{\substack{j=1 \\ j\neq i_0}}^\ell\beta_{i_0j}\irn|\vp|^p|u_j|^p =\sum_{j\in N_h(i_0)}\beta_{i_0j}\irn|u_{i_0^*}|^p|u_j|^p-\sum_{\substack{k=1 \\ k\neq h}}^q\sum_{j\in I_k}|\beta_{i_0j}|\irn|u_{i_0^*}|^p|u_j|^p \\
&\qquad\geq \Big(\min_{\{i,j\}\in E_h}\beta_{ij}\Big)\left[\frac{S^{\bf\phi}}{\dsum_{(i,j)\in\cI_h}\beta_{ij}}\right]^\frac{p}{p-1}-\dsum_{(i,j)\in\cK_h}|\beta_{ij}|\left[\frac{pd_{\bf\phi}}{(p-1)S^{\bf\phi}}\right]^p\Big(\min_{h=1,\ldots,q}\max_{i\in I_h}\beta_{ii}\Big)^{-\frac{p}{p-1}}>0.
\end{align*}
But then, for any $q$, Lemma \ref{lem:cp} implies that $\bf t(\eps)\bf u_\eps\in\cN^{\bf\phi}$ and $\cJ^{\bf\phi}(\bf t(\eps)\bf u_\eps)<\cJ^{\bf\phi}(\bf u)=c^{\bf\phi}$ for $\eps$ small enough. This is a contradiction.
\end{proof}

\begin{theorem} \label{thm:fully_nontrivial}
If $(B_1)$, $(B_2)$ and $(B_3^{\bf\phi})$ hold true, then every least energy block-wise nontrivial solution to $(\mathscr S^{\bf\phi})$ is fully nontrivial.
\end{theorem}

\begin{proof}
Let $\bf u=(u_1,\ldots,u_\ell)\in\cN^{\bf\phi}$ be such that $\cJ^{\bf\phi}(\bf u)=c^{\bf\phi}$. Then, for each $h\in\{1,\ldots,q\}$ there exists $j_0\in I_h$ such that $u_{j_0}\neq 0$. Let $i\in I_h$. By assumption $(B_2)$ there exist $j_1,\ldots,j_m=i\in I_h$ such that $j_i\in N_h(j_{i-1})$ for all $i=1,\ldots, m$. Then Lemma \ref{lem:neighbors} implies that $u_i\neq 0$. This shows that every component of $\bar u_h$ is nontrivial for every $h$.
\end{proof}

Our aim now is to obtain some energy bounds for solutions with positive and sign-changing components. We start with the cooperative case.

\section{The cooperative system}
\label{sec:cooperative}

Throughout this section we assume $(B_1)$ and $(B_2)$.
If $\phi:G\to\z_2$ is a continuous homomorphism, we consider the functional $J^\phi:H^1(\rn)^\phi\to\r$ given by
\[J^\phi(w):=\frac{1}{2}\|w\|^2-\frac{1}{2p}\irn|w|^{2p}\]
and the Nehari manifold
\[\cM^\phi:=\Big\{w\in H^1(\rn)^\phi:w\neq 0, \ \|w\|^2=\irn|w|^{2p}\Big\}\]
associated to the problem
\begin{equation} \label{eq:basic_prob}
-\Delta w+w=|w|^{2p-2}w,\qquad w\in H^1(\rn)^\phi.
\end{equation}
Set
\[\mathfrak{c}^\phi:=\inf_{w\in\cM^\phi}J^\phi(w).\]
If $G$ is the trivial group, i.e., if there are no symmetries, we omit the superscript from the notation and set
\[\mathfrak{c}:=\inf_{w\in\cM}J.\]
It is well known that this value is attained at a positive radial solution $\omega$ to the problem
\begin{equation} \label{eq:omega}
-\Delta w+w=|w|^{2p-2}w,\qquad w\in H^1(\rn).
\end{equation}

We also consider the function $F_h:\r^{\ell_h-\ell_{h-1}}\to \r$ given by
\[F_h(\bar s):=\frac{1}{2}\dsum_{i\in I_h}s_i^2-\frac{1}{2p}\dsum_{{(i,j)\in \cI_h}}\beta_{ij}|s_i|^p|s_j|^p,\qquad \bar s=(s_{\ell_{h-1}+1},\ldots,s_{\ell_h}),\]
and the set 
\[M_h:=\Big\{\bar s\in \r^{\ell_h-\ell_{h-1}}:\bar s\neq \bar 0, \ \dsum_{i\in I_h}s_i^2=\dsum_{(i,j)\in \cI_h}\beta_{ij}|s_i|^p|s_j|^p\Big\}.\]
Note that
\[\inf_{M_h}F_h=\dfrac{p-1}{2p}\mu_h,\]
with $\mu_h$ as defined in \eqref{eq:mu_h}. Clearly, $\mu_h>0$ and it is attained at some $\bar t=(t_{\ell_{h-1}+1},\ldots,t_{\ell_h})\in M_h.$ 

In the following statement we use the Notation \ref{notation}.

\begin{proposition} \label{prop:cooperative}
The following statements hold true:
	\begin{itemize}
		\item[$(i)$] If \ $\bar s=(s_{\ell_{h-1}+1},\ldots,s_{\ell_h})\in M_h$ and $w\in\cM^{\phi_h}$, set ${\bar s w:=(s_{\ell_{h-1}+1}w,\ldots,s_{\ell_h}w)}$. Then $\bar s w\in\cN_h^{\bf\phi}$ and 
		\[\cJ_h^{\bf\phi}(\bar s w)=\frac{2p}{p-1} F_h(\bar s)J^{\phi_h}(w).\]
		\item[$(ii)$] $c_h^{\bf\phi}=\mu_h\mathfrak{c}^{\phi_h}$.
		\item[$(iii)$] If $\mathfrak{c}^{\phi_h}$ is attained at $w\in\cM^{\phi_h}$ and $\mu_h$ is attained at $\bar t\in M_h$, then $\bar t w$ is a least energy nontrivial solution to the system $(\mathscr S_h^{\bf\phi})$, and it is fully nontrivial.
	\end{itemize}
\end{proposition}

\begin{proof}
	$(i):$ \ For any $w\in H^1(\rn)^{\phi_h}\smallsetminus\{0\}$ and $\bar s\in\r^{\ell_h-\ell_{h-1}}\smallsetminus\{0\}$  we have that
	\begin{align*}
		\frac{\dsum\limits_{i\in I_h}\| s_iw\|^2}{\Big(\dsum\limits_{(i,j)\in \cI_h}\irn\beta_{ij}| s_iw|^p| s_jw|^p\Big)^\frac{2}{2p}}=\frac{\dsum\limits_{i\in I_h}|s_i|^2}{\Big(\dsum\limits_{(i,j)\in \cI_h}\beta_{ij}|s_i|^p|s_j|^p\Big)^\frac{2}{2p}}\frac{\|w\|^2}{\Big(\irn|w|^{2p}\Big)^\frac{2}{2p}}.
	\end{align*}
	If $\bar s\in M_h$ and $w\in\cM^{\phi_h}$, by definition, it follows that $\bar s w\in\cN_h^{\bf\phi}$. Raising each side of this identity to the power $\dfrac{p}{p-1}$ and multiplying it by $\dfrac{p-1}{2p}$ we obtain
	\[\cJ_h^{\bf\phi}(\bar s w)=\dfrac{2p}{p-1} F_h(\bar s)J^{\phi_h}(w),\]
	as claimed.
	
	$(ii):$ \ From $(i)$ we get $c_h^{\bf\phi}\leq\mu_h\mathfrak{c}^{\phi_h}$. To prove the opposite inequality, set
	\begin{align*}
		S_h^{\bf\phi}:=\inf_{\substack{\bar u_h\in\cH_h^{\bf\phi} \\ \|\bar u_h\|\neq 0}}\frac{\|\bar u_h\|^2}{\Big(\dsum\limits_{{(i,j)\in \cI_h}}\irn\beta_{ij}|u_i|^p|u_j|^p\Big)^\frac{2}{2p}}, \qquad 
		S^{\phi_h}:=\inf_{\substack{w\in {H^1(\rn)}^{\phi_h} \\ w\neq 0}}\frac{\|w\|^2}{\left(\irn|w|^{2p}\right)^\frac{2}{2p}}.
	\end{align*}
	It is straightforward to verify that  \ $c_h^{\bf\phi}=\dfrac{p-1}{2p}(S_h^{\bf\phi})^\frac{p}{p-1}$ \ and \ $\mathfrak{c}^{\phi_h}=\dfrac{p-1}{2p}(S^{\phi_h})^\frac{p}{p-1}$. \ If $(\bar u_n)$ is a sequence in $\cN_h^{\bf\phi}$ such that $\cJ_h^{\bf\phi}(\bar u_n)\to c_h^{\bf\phi}$, then
	\begin{align*}
		S^{\phi_h}\dsum\limits_{i\in I_h}|u_{in}|_{2p}^2 \leq\dsum\limits_{i\in I_h}\|u_{in}\|^2=\sum\limits_{(i,j)\in \cI_h}\irn\beta_{ij}|u_{in}|^p|u_{jn}|^p \leq \dsum\limits_{(i,j)\in \cI_h}\beta_{ij}|u_{in}|_{2p}^p|u_{jn}|_{2p}^p,
	\end{align*}
	where $|\,\cdot\,|_{2p}$ is the norm in $L^{2p}(\rn)$. Therefore,
	\begin{align*}
(\mu_h)^{\frac{p-1}{p}}S^{\phi_h}\leq\frac{\dsum\limits_{i\in I_h}|u_{in}|_{2p}^2}{\Big(\dsum\limits_{(i,j)\in \cI_h}\beta_{ij}|u_{in}|_{2p}^p|u_{jn}|_{2p}^p\Big)^\frac{2}{2p}}S^{\phi_h}\leq\frac{\dsum\limits_{i\in I_h}\|u_{in}\|^2}{\Big(\dsum\limits_{(i,j)\in \cI_h}\irn\beta_{ij}|u_{in}|^p|u_{jn}|^p\Big)^\frac{2}{2p}}= S_h^{\bf\phi}+o_n(1).
	\end{align*}
	As a consequence, $\mu_h\mathfrak{c}^{\phi_h}\leq c_h^{\bf\phi}$, as claimed.
	\smallskip
	
	$(iii):$ \ If follows from $(i)$ and $(ii)$ that $\cJ^{\bf\phi}_h(\bar t w)=\mu_h\mathfrak{c}^{\phi_h}= c_h^{\bf\phi}$. Hence, $\bar t w$ is a least energy nontrivial solution to $(\mathscr S_h^{\bf\phi})$. By Theorem \ref{thm:fully_nontrivial} it is fully nontrivial.	
\end{proof}

\begin{corollary} \label{cor:cooperative_positive}
	If $\phi_h\equiv 1$, then $c_h^G=\mu_h\mathfrak{c}$ and there exists $\bar s=(s_{\ell_{h-1}+1},\ldots,s_{\ell_h})\in (0,\infty)^{\ell_h - \ell_{h-1}}$ such that $\bar s\omega\in\cN_h^G$ and $\cJ_h^G(\bar s\omega)=c_h^G$, where $\omega$ is the positive radial solution to \eqref{eq:omega}.
\end{corollary}

\begin{proof}
	Since $\omega$ is radial, it is $G$-invariant for any $G$. Hence, $\omega\in\cM^G$ and $J^G(\omega)=\mathfrak{c}^G=\mathfrak{c}$. The value $\mu_h$ is attained at $\bar s=(s_{\ell_{h-1}+1},\ldots,s_{\ell_h})\in M_h$ with $s_i\geq 0$, so Proposition \ref{prop:cooperative} implies that $\bar s\omega\in\cN_h^G$, \ $\cJ_h^G(\bar s\omega)=c_h^G=\mu_h\mathfrak{c}$ and $\bar s\omega$ is fully nontrivial. Hence, $s_i>0$ for every $i$.
\end{proof}

The following groups will play an important role in the proofs of our main results.

\begin{examples} \label{example1}
	Let $N\geq 4$. We write $\rn\equiv\cc\times\cc\times\r^{N-4}$ and a point in $\rn$ as $(z_1,z_2,y)\in\cc\times\cc\times\r^{N-4}$. For $m\in\n$, let 
	\[K_m:=\{\mathrm{e}^{2\pi\mathrm{i}j/m}:j=0,\ldots,m-1\},\]
	$G_m$ be the group generated by $K_m\cup\{\tau\}\cup O(N-4)$, acting on $\rn$ as 
	\begin{align*}
		\mathrm{e}^{2\pi\mathrm{i}j/m}(z_1,z_2,y)&:=(\mathrm{e}^{2\pi\mathrm{i}j/m}z_1,\mathrm{e}^{2\pi\mathrm{i}j/m}z_2,y), \qquad\qquad\tau(z_{1},z_{2},y):=(z_{2},z_{1},y),\\
		\alpha(z_1,z_2,y)&:=(z_1,z_2,\alpha y)\quad\text{if \ }\alpha\in O(N-4).
	\end{align*}
	and $\theta:G_m\to\z_2$ be the homomorphism satisfying $\theta(\mathrm{e}^{2\pi\mathrm{i}j/m})=1$, $\theta(\tau)=-1$, and $\theta(\alpha)=1$ for every $\alpha\in O(N-4)$.
	
	Let $G'_m$ be the subgroup of $G_m$ generated by $K_m\cup\{\tau\}$. Abusing notation we write $\theta:G'_m\to\z_2$ for the restriction of $\theta$ to $G_m'$.
	
	If $\phi_h\equiv 1$ we set $\zeta:=(\frac{1}{\sqrt{2}},\frac{1}{\sqrt{2}},0)$ and, for each $R>1$, we define
	$$\widehat\sigma_{hR}(x):=\sum_{g\in K_m}\omega(x-Rg\zeta),\qquad x\in\rn,$$
and if $\phi_h=\theta$ we take $\zeta:=(1,0,0)$ and we define
	\[\widehat\sigma_{hR}(x):=\sum_{g\in G_m'}\phi_h(g)\,\omega(x-Rg\zeta),\qquad x\in\rn,\]
	where $\omega$ is the positive radial solution to \eqref{eq:omega}. Note that $\widehat\sigma_{hR}(gx)=\phi_h(g)\widehat\sigma_{hR}(x)$ for every $g\in G_m$, $x\in\rn$. Let \ $t_{hR}>0$ \ be such that
	\[\sigma_{hR}:=t_{hR}\widehat\sigma_{hR}\in\cM^{\phi_h}.\]
\end{examples}

\begin{lemma} \label{lem:example}
	If $m\geq 2$ and $m\geq 5$ whenever $\phi_h=\theta$, then there exist $C_0, R_0>0$ such that
	\[J^{\phi_h}(\sigma_{hR})\leq |G'_m\zeta|\,\mathfrak{c}-C_0\mathrm{e}^{-dR}\qquad\forall \ R\geq R_0,\]
	where $d:=|\zeta-\mathrm{e}^{2\pi\mathrm{i}/m}\zeta|$.
\end{lemma}

\begin{proof}
	Consider the case $\phi_h=\theta$ and $\zeta=(1,0,0)$, and set $G:=G_m'$. The $G$-orbit of $\zeta$ is 
	\[G\zeta=\{\zeta_1^+,\ldots,\zeta_m^+,\zeta_1^-,\ldots,\zeta_m^-\}\quad\text{where \ }\zeta_j^+:=\mathrm{e}^{2\pi\mathrm{i}j/m}\zeta,\quad \zeta_j^-:=\tau\zeta_j^+,\]
	so
	\[\widehat\sigma_{hR}(x):=\sum_{j=1}^m\left(\omega(x-R\zeta_j^+)-\omega(x-R\zeta_j^-)\right).\]
	For $y\in\rn$ set
	\[\Psi(y):=\irn\omega^{2p-1}(x)\omega(x-y)\d x.\]
	From the well known decay estimates for $\omega$ we deduce that
	\[\lim_{y\to\infty}\Psi(y)|y|^{\frac{N-1}{2}}\exp|y|=b>0.\]
	Hence, for large enough $R$,
	\[\frac{\Psi(Ry)}{\Psi(Rz)}\leq C\mathrm{e}^{-R(|y|-|z|)}\qquad\text{if \ }|y|\geq|z|.\]
	Here, and in what follows, $C$ denotes a positive constant independent of $R$. Set
	\begin{align*}
		&\eps_R^+:=\sum\limits_{\substack{i,j=1 \\ i\neq j}}^m\Psi(R\zeta_i^+-R\zeta_j^+),\qquad\eps_R^-:=\sum\limits_{\substack{i,j=1 \\ i\neq j}}^m\Psi(R\zeta_i^--R\zeta_j^-),\\
		&\tilde\eps_R:=\sum_{i,j=1}^m\Psi(R\zeta_i^+-R\zeta_j^-).
	\end{align*}
	Then, $\eps_R^+=\eps_R^-$. Since $m\geq 5$, we have that $|\zeta_i^+-\zeta_j^-|>|\zeta_1^+-\zeta_2^+|=d$ for every $i,j=1,\ldots,m$. Therefore, 
	\[\frac{\tilde\eps_R}{\eps_R^+}\leq\sum_{i,j=1}^m\frac{\Psi(R\zeta_i^+ -R\zeta_j^-)}{\Psi(R\zeta_1^+-R\zeta_2^+)}\leq C \sum_{i,j=1}^m\mathrm{e}^{-R(|\zeta_i^+-\zeta_j^-|-d)}\leq C\mathrm{e}^{-Ra},\]
	where $a:=\min\limits_{i,j=1,\ldots,m}|\zeta_i^+-\zeta_j^-|-d>0$. This shows that $\tilde\eps_R=o(\eps_R^+)$ as $R\to\infty$.
	Since $\omega$ solves \eqref{eq:omega}, for any $y,y'\in\rn$ we have that
	\begin{align*}
		&  \irn\left[\nabla\omega(x-y)\cdot\nabla\omega(x-y')+\omega(x-y)\omega(x-y')\right]\d x\\
		& \qquad =\irn\omega^{2p-1}(x-y)\omega(x-y')\d x=\Psi(y'-y).
	\end{align*}
	So, using \cite[Lemma 4]{cc}, we get
	\begin{align*}
		\frac{\|\widehat\sigma_{hR}\|^{2}}{|\widehat\sigma_{hR}|_{2p}^{2}} &\leq \frac{2m\|\omega\|^2+2\eps_R^+ +o(\eps_R^+)}{\left(2m\|\omega\|^2+2(2p-1)\eps_R^+ +o(\eps_R^+)\right)^\frac{2}{2p}}\leq(2m\|\omega\|^2)^{\frac{p-1}{p}}-C\eps_R^+.
	\end{align*}
	As a consequence,
	\[J^\phi(\sigma_{hR})=\dfrac{p-1}{2p}\Big(\frac{\|\widehat\sigma_{hR}\|^{2}}{|\widehat\sigma_{hR}|_{2p}^{2}}\Big)^\frac{p}{p-1}\leq 2m\,\mathfrak{c}-C_0\mathrm{e}^{-dR},\]
	as claimed.
	
	The proof for $\phi_h\equiv 1$ and $\zeta=(\frac{1}{\sqrt{2}},\frac{1}{\sqrt{2}},0)$ is similar but simpler.
\end{proof}

\begin{proposition} \label{prop:cooperative_nodal}
	Let $G'_m$ and $\theta$ be as in \emph{Example \ref{example1}}. If $m\geq 5$, then
	\begin{itemize}
		\item[$(a)$] the problem \eqref{eq:basic_prob} has a least energy nontrivial solution $\widehat{\omega}$ and \ $2\,\mathfrak{c}<\mathfrak{c}^\theta<2m\,\mathfrak{c}$.
		\item[$(b)$] Let $\phi_h:=\theta$. Then the system $(\mathscr S_h^{\bf\phi})$ has a least energy solution of the form $(t_{\ell_{h-1}+1}\widehat{\omega},\ldots,t_{\ell_h}\widehat{\omega})$ with $t_i>0$ for all $i\in I_h$, and its energy satisfies
		\[2\,\mu_h\mathfrak{c}<c_h^{\bf\phi}<2m\,\mu_h\mathfrak{c}.\]
	\end{itemize}
\end{proposition}

\begin{proof}
	$(a):$ \ Fix $m\geq 5$ and set $G:=G_m$. Let $\xi=(z_1,z_2,y)\in\cc\times\cc\times\r^{N-4}$. Then $|G\xi|=m$ and $G_\xi=\{1,\tau\}$ if $z_1=z_2\neq 0$, and \ $|G\xi|=2m$ and $G_\xi=\{1\}$ if $z_1\neq z_2$.
	According to \cite[Corollary 3.1]{CSr}, $\mathfrak{c}^\theta$ is attained if 
	\[\mathfrak{c}^\theta<|G\xi|\,\mathfrak{c}^{\theta|G_\xi}\quad\forall \ \xi=(z_1,z_2,y)\in\cc\times\cc\times\r^{N-4}\text{ \ with \ }(z_1,z_2)\neq (0,0).\]
	i.e., if \ $\mathfrak{c}^\theta<\min\{2m\mathfrak{c},\,m\mathfrak{c}^{\theta|\{1,\tau\}}\}.$ \ 
	Note that every $w\in\cM^{\theta|\{1,\tau\}}$ changes sign and $w^+:=\max\{w,0\}$ and $w^-:=\min\{w,0\}$ belong to $\cM$. Therefore,
	\[J^{\theta|\{1,\tau\}}(w)=\dfrac{p-1}{2p}\|w\|^2=\dfrac{p-1}{2p}\|w^+\|^2+\dfrac{p-1}{2p}\|w^-\|^2\geq 2\mathfrak{c}.\]
	As a consequence, $\min\{2m\mathfrak{c},\,m\mathfrak{c}^{\theta|\{1,\tau\}}\}=2m\mathfrak{c}$. 
	Let $\sigma_{hR}\in\cM^\theta$ be as in Example \ref{example1}. By Lemma \ref{lem:example},
	\[\mathfrak{c}^\theta\leq J^\theta(\sigma_{hR})<2m\,\mathfrak{c}.\]
	It is well known that the energy of any sign-changing solution to \eqref{eq:omega} is greater than $2\,\mathfrak{c}$. This proves $(a)$.
	
	$(b):$ \ This statement is easily derived from $(a)$ and Proposition \ref{prop:cooperative}.
\end{proof}

\begin{remark}
	\emph{We wish to stress that the previous proposition is valid (with $2m$ replaced by $|G\zeta|$) for any group $G$ and any homomorphism $\phi_h:G\to\z_2$ for which there exist $\zeta\in\rn$ and $g\in\ker\phi_h$ such that
		\[\min\{|\bar g\zeta-\zeta|:\bar g\in G, \ \phi_h(\bar g)=-1\}>|g\zeta-\zeta|>0.\]
		Indeed, defining $\sigma_{hR}$ as in Example \ref{example1}, it is easy to see that the proof of Lemma \ref{lem:example} relies only on this property. }
\end{remark}

\section{Energy bounds for a system with two blocks}
\label{sec:two_blocks}

In this section we consider the system $(\mathscr S^{\bf\phi})$ with two blocks, i.e., $q=2$. As before, we assume $(B_1)$ and $(B_2)$. We have the following result.

\begin{lemma} \label{lem:energy_bounds}
	Let $N=4$ or $N\geq 6$, \ $q=2$, \ $G:=G_m$ and $\theta$ be as in \emph{Examples }\ref{example1} and, for $h=1,2$, let $\phi_h$ be either the trivial homomorphism or equal to $\theta$. If $m\geq 6$, then
	\begin{equation*}
		c^{\bf\phi}<c_k^{\bf\phi} + |G\xi|\,c_h^{{\bf\phi}|G_\xi}\qquad\text{for every \ }\xi\in\rn\smallsetminus\{0\}\text{ \ and \ }k\neq h, \ h,k\in\{1,2\}.
	\end{equation*}
\end{lemma}

\begin{proof}
	Fix $m\geq 6$. Let $\xi=(z_1,z_2,y)\in\cc\times\cc\times\r^{N-4}$. Then $|G\xi|=\infty$ if $y\neq 0$, \ $|G\xi|=m$ and $G_\xi=\{1,\tau\}$ if $z_1=z_2\neq 0$ and $y=0$, and \ $|G\xi|=2m$ and $G_\xi=\{1\}$ if $z_1\neq z_2$ and $y=0$. As in the proof of Proposition \ref{prop:cooperative_nodal} we see that $\min\{2m\mathfrak{c},\,m\mathfrak{c}^{\theta\,|\{1,\tau\}}\}=2m\mathfrak{c}$, so Proposition \ref{prop:cooperative} yields
	\begin{equation} \label{eq:1}
	\min_{\xi\in\rn\smallsetminus\{0\}} |G\xi|\,c_h^{{\bf\phi}|G_\xi}=\min_{\xi\in\rn\smallsetminus\{0\}} |G\xi|\mu_h\mathfrak{c}^{{\phi_h}|G_\xi}=
		\begin{cases}
			m\mu_h\mathfrak{c} &\text{if \ }\phi_h\equiv 1,\\ \smallskip
			2m\mu_h\mathfrak{c} &\text{if \ }\phi_h=\theta.
		\end{cases}
		\end{equation}
	
Without loss of generality, assume $k=1$ and $h=2$. Let $w_1$ be a least energy solution to the problem 
\[-\Delta w+w=|w|^{2p-2}w,\qquad w\in H^1(\rn)^{\phi_1},\]	
i.e., $w_1$ is the positive radial solution $\omega$ to \eqref{eq:omega} if $\phi_1\equiv 1$, and $w_1=\widehat\omega$ as in Proposition \ref{prop:cooperative_nodal} if $\phi_1=\theta$. By Corollary \ref{cor:cooperative_positive} and Proposition \ref{prop:cooperative_nodal} there exists $\bar t_1\in (0,\infty)^{\ell_1}$ be such that $\bar t_1w_1\in\cN_1^{\bf\phi}$ and $\cJ_1^{\bf\phi}(\bar t_1w_1)=c_1^\phi$. 

Let $\zeta:=(\frac{1}{\sqrt{2}},\frac{1}{\sqrt{2}},0)$ if $\phi_2\equiv 1$ and $\zeta:=(1,0,0)$ if $\phi_2=\theta$, and for each $R>1$ define $\sigma_{2R}\in\cM^{\phi_2}$ as in Examples \ref{example1}. Let $\bar t_2\in M_2$ be such that $F_2(\bar t_2)=\mu_2$. Then $\bar t_2\sigma_{2R}\in\cN_2^{\bf\phi}$, and from Proposition \ref{prop:cooperative} and Lemma \ref{lem:example} for $R$ large enough we get
\begin{align*}
\cJ_2^{\bf\phi}(\bar t_2\sigma_{2R}) =\mu_2J^{\phi_2}(\sigma_{2R})\leq \mu_2(|G\zeta|\mathfrak{c}-C_0\mathrm{e}^{-dR}),
\end{align*}
where $d:=|\zeta-\mathrm{e}^{2\pi\mathrm{i}/m}\zeta|$. 

Define $\bf w_R=(w_{1R},\ldots,w_{\ell R})=(\bar w_{1R},\bar w_{2R})$ with $\bar w_{1R}:=\bar t_1w_1$ and $\bar w_{2R}:=\bar t_2\sigma_{2R}$. Then $\bf w_R$ satisfies \eqref{eq:N} if $R$ is large enough, so there exist $R_0>0$ and, for each $R>R_0$, a pair $(s_{1R},s_{2R})\in[{1}/{2},2]^2$ such that $\bf u_R:=(s_{1R}\bar w_{1R},s_{2 R}\bar w_{2 R})\in\cN^{\bf\phi}$. From the last statement in Lemma \ref{lem:N}$(ii)$ we derive
	\begin{align*}
\cJ^{\bf\phi}(\bf u_R) &= \frac{1}{2}\sum_{i=1}^\ell\|s_{iR}w_{iR}\|^2 - \frac{1}{2p}\sum_{i,j=1}^\ell\beta_{ij}\irn |s_{iR}w_{iR}|^p|s_{jR}w_{jR}|^p \\
&\leq \cJ_1^{\bf\phi}(\bar w_{1R})+\cJ_2^{\bf\phi}(\bar w_{2R}) - \frac{1}{2p}\sum_{i\in I_1, \, j\in I_2}\beta_{ij}\irn |s_{1R}w_{iR}|^p|s_{2R}w_{jR}|^p \\
&\leq c_1^{\bf\phi} + |G\zeta|\mu_2\mathfrak{c} - C\mathrm{e}^{-Rd}  - \frac{1}{2p}\sum_{i\in I_1, \, j\in I_2}\beta_{ij}\irn |s_{1R}w_{iR}|^p|s_{2R}w_{jR}|^p.
	\end{align*}
Since $\omega$ and $\widehat{\omega}$ solve \eqref{eq:omega} we have that
	\begin{equation*}
		|w_{iR}(x)|\leq C\mathrm{e}^{-|x|}\quad\forall \ i\in I_1\qquad\text{and}\qquad |w_{jR}(x)|\leq C\mathrm{e}^{-|x-R\zeta|}\quad\forall \ j\in I_2,
	\end{equation*}
see \cite[Corollary 2.4]{ad}. Therefore,
\[\irn |w_{iR}|^p|w_{jR}|^p \leq C\irn \mathrm{e}^{-p|x|}\,\mathrm{e}^{-p|x-R\zeta|}\d x \leq C\mathrm{e}^{-Rp}\qquad\forall \ i\in I_1, \ j\in I_2.\]
Since $p>1\geq d$ when $m\geq 6$, it follows that
\[c^{\bf\phi}\leq \cJ^{\bf\phi}(\bf u_R)<c_1^{\bf\phi} + |G\zeta|\mu_2\mathfrak{c}.\]
This, together with \eqref{eq:1}, yields our claim.
\end{proof}

\begin{remark}\label{rem:bounds}
\emph{To extend the previous argument to a system with more than two blocks one needs some information on the decay of the components of solutions to a system. Namely, if we knew that a minimizer $(w_1,\ldots,w_{\ell_{q-1}})$ of $\cJ_Q^{\bf\phi}$ on $\cN_Q^{\bf\phi}$ with $Q=\{1,\ldots,q-1\}$, $q\geq 3$, satisfies
\begin{equation*}
|w_{i}(x)|\leq C_\delta\mathrm{e}^{-\delta|x|}\qquad\forall \ i\in I_Q,
\end{equation*}
for every $\delta\in(0,1)$ and some constant $C_\delta$, then, taking $w_{jR}$ to be a suitable multiple of $\sigma_{qR}$ for $j\in I_q$ and choosing $\delta$ such that $p\delta>1$, we would have
\[\irn |w_{i}|^p|w_{jR}|^p \leq C\irn \mathrm{e}^{-p\delta|x|}\,\mathrm{e}^{-p|x-R\zeta|}\d x \leq C\mathrm{e}^{-Rp\delta}= o(\mathrm{e}^{-Rd})\qquad\forall \ i\in I_Q, \ j\in I_q.\]
This would allow us to extend Lemma \ref{lem:energy_bounds} to any $q>2$. Information on the decay of the components does not seem to be available so far.}
\end{remark}

\begin{theorem} \label{thm:energy_estimates}
Let $N=4$ or $N\geq 6$, $q=2$, $G:=G_6$ and $\theta$ be as in \emph{Example} \ref{example1} (with $m=6$). Let $\{1,2\}=Q^+\cup Q^-$ with $Q^+\cap Q^-=\emptyset$ and set $\phi_h\equiv 1$ if $h\in Q^+$ and $\phi_h=\theta'$ if $h\in Q^-$. Then there is a least energy block-wise nontrivial solution $\bf w=(\bar w_1,\bar w_2)$ to the system $(\mathscr S^{\bf\phi})$ such that all nontrivial components of $\bar w_h$ are positive if $h\in Q^+$, all nontrivial components of $\bar w_h$ change sign if $h\in Q^-$, and
\begin{equation} \label{eq:upper_bds2}
\cJ^{\bf\phi}(\bf w)<
\begin{cases}
\min\limits_{h\neq k}(\mu_k+6\mu_h)\mathfrak{c} &\text{if \ }Q^+=\{1,2\}, \\
(\mu_k+12\mu_h)\mathfrak{c} &\text{if \ }k\in Q^+\text{ and }h\in Q^-,\smallskip \\
12(\mu_1+\mu_2)\mathfrak{c} &\text{if \ }Q^-=\{1,2\}.
\end{cases}
\end{equation}
\end{theorem}

\begin{proof}
Lemma \ref{lem:energy_bounds} with $m=6$, $\phi_h\equiv 1$ if $h\in Q^+$ and $\phi_h=\theta$ if $h\in Q^-$, together with Corollary \ref{cor:main_inequality} and Remark \ref{rem:positive}, yields the existence of a solution $\bf w=(\bar w_1,\bar w_2)$ such that all nontrivial of $\bar w_h$ are positive if $h\in {Q^+}$, all nontrivial components of $ \bar w_k$ change sign if $k\in {Q^-}$, and
\[\cJ^{\bf\phi}(\bf w)=c^{\bf\phi}<c_k^{\bf\phi} + |G\xi|\,c_h^{{\bf\phi}|G_\xi}\qquad\text{for every \ }\xi\in\rn\smallsetminus\{0\}\text{ \ and \ }k\neq h, \ h,k\in\{1,2\}.\]
The upper bounds for $\cJ^{\bf\phi}(\bf w)$ now follow from Corollary \ref{cor:cooperative_positive}, Proposition \ref{prop:cooperative_nodal} and equation \eqref{eq:1}.
\end{proof}
\medskip

\begin{proof}[Proof of Theorem \ref{thm:Main1}.] \ Let $\phi_h\equiv 1$ if $h\in Q^+$ and $\phi_h$ be equal to $\theta:G'_m\to\z_2$ if $q=1$, $\theta:G_m\to\z_2$ if $q=2$ and $\t:G_\infty\to\z_2$ if $q\geq 3$, as in Examples \ref{example1} and \ref{example2}, if $h\in Q^-$. Set $C_q:=C_{\bf\phi}$ as in \eqref{eq:C*}. Then, Corollary \ref{cor:cooperative_positive}, Proposition \ref{prop:cooperative_nodal}, Theorem \ref{thm:energy_estimates} and Theorem \ref{thm:existence} establish existence of a least energy block-wise nontrivial solution whose components have the required signs, and provide the upper energy estimates for $q=1,2$. Theorem \ref{thm:fully_nontrivial} asserts that these solutions are fully nontrivial. The strict lower energy estimate follows from Proposition \ref{prop:ground_states}, Corollary \ref{cor:cooperative_positive} and Proposition \ref{prop:cooperative_nodal}.
\end{proof}

\section{The singularly perturbed system}
\label{sec:singularly_perturbed}

Let $G$ be a closed subgroup of $O(N)$, $\phi_h:G\to\z_2$ be continuous homomorphisms satisfying $(A_{\phi_h})$ for each $h\in Q$ and define $\phi_i := \phi_h$ for all $i\in I_h$. Let $\o$ be a $G$-invariant bounded domain in $\rn$ such that $0\in\o$. We consider the system of singularly perturbed elliptic equations
\begin{equation*}
(\mathscr S_{\eps,\o}^{\bf\phi})\qquad
\begin{cases}
-\eps^2\Delta u_i+ u_i = \dsum_{j=1}^\ell \beta_{ij}|u_j|^p|u_i|^{p-2}u_i, \\
u_i\in H^1(\o)^{\phi_i},\qquad i=1,\ldots,\ell,
\end{cases}
\end{equation*}
where $\eps>0$, $1<p<\frac{N}{N-2}$, $N\geq 4$, and the matrix $(\beta_{ij})$ satisfies assumption $(B_1)$.

Note that $\bf u=(u_1,\ldots,u_\ell)$ solves $(\mathscr S_{\eps,\o}^{\bf\phi})$ iff $\wt{\bf u}=(\wt{u}_1,\ldots,\wt{u}_\ell)$ with $\wt{u}_i(z):=u_i(\eps z)$ solves the system
\begin{equation*}
(\mathscr{S}_{\o_\eps}^{\bf\phi})\qquad
\begin{cases}
-\Delta v_i+v_i = \dsum_{j=1}^\ell \beta_{ij}|v_j|^p|v_i|^{p-2}v_i, \\
v_i\in H^1(\o_\eps)^{\phi_i},\qquad i=1,\ldots,\ell,
\end{cases}
\end{equation*}
where $\o_\eps:=\{z\in\rn:\eps z\in\o\}$.

A \textbf{least energy block-wise nontrivial solution to this system} is defined as in Section \ref{sec:preliminaries}, replacing $\rn$ with $\o_\eps$. If $\eps_n\to 0$ and $\wt{\bf u}_n$ is a least energy block-wise nontrivial solution to $(\mathscr{S}_{\o_{\eps_n}}^{\bf\phi})$, then, arguing as in \cite[Lemma 4.2]{CS}, one sees that  $(\wt{\bf u}_n)$ is a minimizing sequence for the functional $\cJ^{\bf\phi}$ associated to the system $(\mathscr S^{\bf\phi})$ on $\cN^{\bf\phi}$.

\begin{proposition} \label{prop:eps_fully_nontrivial}
Assume $(B_1)$ and $(B_2)$. The following statements hold true:
\begin{itemize}
\item[$(i)$] For every $\eps>0$ the system $(\mathscr S_{\o_\eps}^{\bf\phi})$ has a least energy block-wise nontrivial solution.
\item[$(ii)$] If $\F(G)\neq 0$ there exists $\eps_0>0$ such that, if $\eps<\eps_0$, every least energy block-wise nontrivial solution to the system $(\mathscr S_{\o_\eps}^{\bf\phi})$ is fully nontrivial.
\end{itemize}
\end{proposition}

\begin{proof}
$(i):$ It follows from \cite[Theorem 1.1]{cp} that $(\mathscr{S}_{\o_\eps}^{\bf\phi})$ has a least energy blockwise nontrivial solution.

$(ii):$ Arguing by contradiction, assume that for any $n\in\n$ there exist $\eps_n\in(0,{1/n})$ and a least energy block-wise nontrivial solution $\wt{\bf u}_n$ to $(\mathscr{S}_{\o_{\eps_n}}^{\bf\phi})$ which is not fully nontrivial. Passing to a subsequence, we may choose $i\in I_k$ such that $\wt{u}_{in}=0$ for all $n$. 

As stated above, $(\wt{\bf u}_n)$ is a minimizing sequence for $\cJ^{\bf\phi}$ on $\cN^{\bf\phi}$. If $\F(G)\neq 0$ Proposition \ref{prop:ground_states} states that no subsystem of $(\mathscr S^{\bf\phi})$ having at least two blocks has a least energy block-wise nontrivial solution. So we derive from Theorem \ref{thm:splitting} that, after passing to a subsequence, there exist a sequence $(\xi_{kn})$ in $\rn$, a closed subgroup $G_k$ of $G$, and a least energy nontrivial solution $\bar v_k$ to the cooperative system
\begin{equation*}
(\mathscr S_{k}^{\bf\phi|G_k})\qquad
\begin{cases}
-\Delta v_i+ v_i = \dsum\limits_{j\in I_k} \beta_{ij}|v_j|^p|v_i|^{p-2}v_i, \\
v_i\in H^1(\rn)^{\phi_i|G_k},\qquad i\in I_k,
\end{cases}
\end{equation*}
such that $|G/G_k|<\infty$, $\dlim\limits_{n\to\infty}|g\xi_{kn}-\bar g\xi_{kn}|=\infty$ for any $g,\bar g\in G$ with $\bar g g^{-1}\notin G_k$, and
\[\lim_{n\to\infty}\Big\|\wt u_{in}-\dsum_{[g]\in G/G_k}\phi_i(g)(v_i\circ g^{-1})( \ \cdot \ -g\xi_{kn})\Big\|=0.\]
Since we are assuming that $\wt u_{in}=0$ for all $n$, this implies that $v_i=0$, which is a contradiction because $\bar v_k$ is fully nontrivial by Theorem \ref{thm:fully_nontrivial}.
\end{proof}

Recall Examples \ref{example1}. Note that $\F(G'_m)\neq \{0\}$ if $N\geq 5$. The following theorem establishes the existence of solutions whose limit profile becomes uncoupled.

\begin{theorem} \label{thm:eps_uncoupled}
Assume $(B_1)$ and $(B_2)$. Let $N\geq 5$. Set $G':=G'_5$ and let $\o$ be a $G$-invariant bounded domain in $\rn$ such that $0\in\o$. Write $\{1,\ldots,q\}=Q^+\cup Q^-$ with $Q^+\cap Q^-=\emptyset$ and let $\phi_h\equiv 1$ if $h\in Q^+$ and $\phi_h=\theta$ if $h\in Q^-$. Then, for any $\eps_n>0$ with $\eps_n\to 0$, after passing to a subsequence, there exists a least energy fully nontrivial solution $\wt{\bf u}_n=(\wt{u}_{1n},\ldots,\wt{u}_{\ell n})$ to the system $(\mathscr{S}_{\o_{\eps_n}}^{\bf\phi})$ such that $\wt{u}_{in}$ is positive if $i\in I_{Q^+}$ and $\wt{u}_{in}$ changes sign if $i\in I_{Q^-}$. Furthermore, for each $h=1,\ldots,q$, there exist a sequence $(\xi_{hn})$ in $\rn$ and a least energy fully nontrivial solution $\bar v_h$ to the cooperative system
\begin{equation*}\label{eq:subsystem_h_phi}
\tag{$\mathscr S_{h}^{\bf\phi}$}\qquad
\begin{cases}
-\Delta v_i+ v_i = \dsum\limits_{j\in I_h} \beta_{ij}|v_j|^p|v_i|^{p-2}v_i, \\
v_i\in H^1(\rn)^{\phi_h},\qquad i\in I_h,
\end{cases}
\end{equation*}
which satisfy
\begin{itemize}
\item[$(a)$] the components of $\bar v_h$ are positive if $h\in Q^+$ and they change sign if $h\in Q^-$,
\item[$(b)$] $\xi_{hn}\in\F(G')$ for all $h$, and $\dlim_{n\to\infty}|\xi_{hn}-\xi_{kn}|=\infty$ if $h\neq k$, 
\item[$(c)$] $\dlim_{n\to\infty}\left\|\wt u_{in}-v_i( \ \cdot \ -\xi_{hn})\right\|=0$ \ for every \ $i\in I_h,$
\item[$(d)$] $\|\bar v_h\|^2=\mu_h\|\omega\|^2$ \ if $h\in Q^+$, \ and \ $\|\bar v_h\|^2<10\,\mu_h\|\omega\|^2$ \ if $h\in Q^-$.
\end{itemize}
\end{theorem}

\begin{proof}
Let $\wt{\bf u}_n$ be a least energy block-wise nontrivial solution to $(\mathscr S^{\bf\phi}_{\o_{\eps_n}})$. Note that $\F(G')\neq 0$ because $N\geq 5$. So, after passing to a subsequence, $\wt{\bf u}_n$ is fully nontrivial by Proposition \ref{prop:eps_fully_nontrivial} and, invoking Proposition \ref{prop:ground_states}, we derive from Theorem \ref{thm:splitting} that, for each $h=1,\ldots,q$ there exist a sequence $(\xi_{hn})$ in $\rn$, a closed subgroup $G'_h$ of $G'$, and a least energy fully nontrivial solution $\bar v_h$ to the cooperative system
\begin{equation}
\tag{$\mathscr S_{h}^{\bf\phi|G_h}$}\qquad
\begin{cases}
-\Delta v_i+ v_i = \dsum\limits_{j\in I_h} \beta_{ij}|v_j|^p|v_i|^{p-2}v_i, \\
v_i\in H^1(\rn)^{\phi_h|G'_h},\qquad i\in I_h,
\end{cases}
\end{equation}
satisfying properties $(ii)$, $(iv)$ and $(v)$. Property $(v)$ combined with Proposition \ref{prop:ground_states} yields
\[c^{\bf\phi} = \sum_{h=1}^q|G'/G'_h|\,c^{\bf\phi|G'_h}_{h}= \sum_{h=1}^qc^{\bf\phi}_{h}.\]
If $h\in Q^+$, then $\phi_h\equiv 1$ and Corollary \ref{cor:cooperative_positive} gives $c^{\bf\phi|G'_h}_{h}=c^{\bf\phi}_{h}$. On the other hand, if $h\in Q^-$, then  $\phi_h=\theta$ and arguing as in  Proposition \ref{prop:cooperative_nodal} we see that 
\[\mathfrak{c}^{\theta}<|G'\xi|\,\mathfrak{c}^{\theta|G'_\xi}\quad\forall \ \xi=(z_1,z_2,y)\in\cc\times\cc\times\r^{N-4}\text{ \ with \ }(z_1,z_2)\neq (0,0).\]
Then, Proposition \ref{prop:cooperative} yields $c_h^{\bf\phi}<|G'/G'_h|\,c^{\bf\phi|G'_h}_{h}$ if $G'_h\neq G'$. As a consequence, $G'_h=G'$ for every $h$. So $v_i$ changes sign for every $i\in I_{Q^-}$. Furthermore, $\xi_{hn}\in\F(G')$, so from property $(ii)$ in Theorem \ref{thm:splitting} we get that $|\xi_{hn}-\xi_{kn}|\to\infty$ if $h\neq k$, and property $(iv)$ becomes 
\[\lim_{n\to\infty}\left\|\wt u_{in}-v_i( \ \cdot \ -\xi_{hn})\right\|=0\qquad\text{for every \ }i\in I_h.\]
Since $\dfrac{p-1}{2p}\|\bar v_h\|^2=c_h^{\bf \phi}$ and $\mathfrak{c} = \dfrac{p-1}{2p}\|\omega\|^2$, the statement $(d)$ follows from the estimates in Corollary \ref{cor:cooperative_positive} and Proposition \ref{prop:cooperative_nodal}.
\end{proof}

\begin{proof}[Proof of Theorem \ref{thm:Main2}.] This result is easily derived from Theorem \ref{thm:eps_uncoupled}, performing the change of variable described at the beginning of this section.
\end{proof}

The following two theorems establish the existence of solutions whose limit profile remains coupled.

\begin{theorem} \label{thm:eps_coupled2}
Assume $(B_1)-(B_3)$. Let $N=4$ or $N\geq 6$, $q=2$, $G:=G_6$ and $\theta$ be as in Example \ref{example1} (with $m=6$) and $\o$ be a $G$-invariant bounded domain in $\rn$ such that $0\in\o$. Let $\{1,2\}=Q^+\cup Q^-$ with $Q^+\cap Q^-=\emptyset$ and set $\phi_h\equiv 1$ if $h\in Q^+$ and $\phi_h=\theta$ if $h\in Q^-$. Then, for any $\eps_n>0$ with $\eps_n\to 0$, after passing to a subsequence, there exist a least energy fully nontrivial solution $\wt{\bf u}_n=(\wt{u}_{1n},\ldots,\wt{u}_{\ell n})$ to the system $(\mathscr{S}_{\o_{\eps_n}}^{\bf\phi})$ with $\wt{u}_{in}$ positive if $i\in I_{Q^+}$ and $\wt{u}_{in}$ sign-changing if $i\in I_{Q^-}$, and a least energy fully nontrivial solution $\bf w=(\bar w_1,\bar w_2)$ to the system
\begin{equation*}
\tag{$\mathscr S^{\bf\phi}$}\qquad
\begin{cases}
-\Delta v_i+ v_i = \dsum\limits_{j\in I_h} \beta_{ij}|v_j|^p|v_i|^{p-2}v_i, \\
v_i\in H^1(\rn)^{\phi_h},\qquad i\in I_h,\qquad h=1,2,
\end{cases}
\end{equation*}
such that
\begin{itemize}
\item[$(a)$] the components of $\bar w_h$ are positive if $h\in Q^+$ and they change sign if $h\in Q^-$,
\item[$(b)$] $\dlim_{n\to\infty}\left\|\wt u_{in}-w_i\right\|=0$ \ for every \ $i\in I_{\{1,2\}}$,
\item[$(c)$] $\bf w$ satisfies the energy estimates \eqref{eq:upper_bds2}.
\end{itemize}
\end{theorem}

\begin{proof}
Let $\wt{\bf u}_n$ be a least energy block-wise nontrivial solution to $(\mathscr S^{\bf\phi}_{\o_{\eps_n}})$. By Theorem \ref{thm:fully_nontrivial} it is fully nontrivial. Comparing Lemma \ref{lem:energy_bounds} with property $(v)$ of Theorem \ref{thm:splitting} we see that the set $Q$ given by that theorem must be all of $\{1,2\}$. Hence, there exists a solution $\bf w$, as claimed, to the system $(\mathscr{S}^{\bf\phi})$ satisfying $(a)$ and $(b)$. The energy estimates $(c)$ are given by Theorem \ref{thm:energy_estimates}.
\end{proof}

\begin{theorem} \label{thm:eps_coupled}
Assume $(B_1)-(B_3)$. Let $N=4$ or $N\geq 6$, $G:=G_\infty$ and $\t$ be as in \emph{Example} \ref{example2} and $\o$ be a $G$-invariant bounded domain in $\rn$ such that $0\in\o$. Let $\{1,\ldots,q\}=Q^+\cup Q^-$ with $Q^+\cap Q^-=\emptyset$ and set $\phi_h\equiv 1$ if $h\in Q^+$ and $\phi_h=\t$ if $h\in Q^-$. Then, for any $\eps_n>0$ with $\eps_n\to 0$, after passing to a subsequence, there exist a least energy fully nontrivial solution $\wt{\bf u}_n=(\wt{u}_{1n},\ldots,\wt{u}_{\ell n})$ to the system $(\mathscr{S}_{\o_{\eps_n}}^{\bf\phi})$ with $\wt{u}_{in}$ positive if $i\in I_{Q^+}$ and $\wt{u}_{in}$ sign-changing if $i\in I_{Q^-}$, and a least energy fully nontrivial solution $\bf w=(\bar w_1,\ldots,\bar w_q)$ to the system
\begin{equation*}
\tag{$\mathscr S^{\bf\phi}$}\qquad
\begin{cases}
-\Delta v_i+ v_i = \dsum\limits_{j\in I_h} \beta_{ij}|v_j|^p|v_i|^{p-2}v_i, \\
v_i\in H^1(\rn)^{\phi_h},\qquad i\in I_h,\qquad h=1,\ldots,q,
\end{cases}
\end{equation*}
satisfying
\begin{itemize}
\item[$(a)$] the components of $\bar w_h$ are positive if $h\in Q^+$ and they change sign if $h\in Q^-$,
\item[$(b)$] $\dlim_{n\to\infty}\left\|\wt u_{in}-w_i\right\|=0$ \ for every \ $i\in I_h$.
\end{itemize}
\end{theorem}

\begin{proof}
Let $\wt{\bf u}_n$ be a least energy block-wise nontrivial solution to $(\mathscr S^{\bf\phi}_{\o_{\eps_n}})$. By Theorem \ref{thm:fully_nontrivial} it is fully nontrivial. Since the only finite $G$-orbit is the origin, property $(ii)$ of Theorem \ref{thm:splitting} says that splitting is impossible. So the set $Q$ given by that theorem must be the whole $\{1,\ldots,q\}$. Hence, there exists a solution $\bf w$ to the system $(\mathscr{S}^{\bf\phi})$ with the stated properties.
\end{proof}
	
\begin{proof}[Proof of Theorem \ref{thm:Main3}.] This result follows immediately from Theorems \ref{thm:eps_coupled2} and \ref{thm:eps_coupled}, performing the change of variable described at the beginning of this section.
\end{proof}

	\bigskip
	
\begin{flushleft}
	
		\textbf{Mónica Clapp} and \textbf{Mayra Soares}\\
		Instituto de Matemáticas\\
		Universidad Nacional Autónoma de México\\               
		Circuito Exterior, Ciudad Universitaria\\
		04510 Coyoacán, Ciudad de México, Mexico\\
		\texttt{monica.clapp@im.unam.mx} \\
		\texttt{ssc\_mayra@im.unam.mx}
	
	\end{flushleft}
	
\end{document}